\documentclass[oneside,english,12pt]{amsart}
\usepackage{amsthm}
\usepackage{hyperref}
\usepackage{txfonts}
\usepackage{amsfonts}
\usepackage{lmodern}
\usepackage{setspace}

\usepackage[T1]{fontenc}
\usepackage[latin9]{inputenc}
\usepackage{geometry}
\usepackage{amsthm}
\usepackage{amssymb}
\usepackage[all]{xy}
\usepackage{pinlabel}
\usepackage{dirtytalk}
\usepackage{amsmath}

\usepackage{mathtools}
\usepackage{amssymb,amsmath,amsthm,tikz}
\usepackage{pinlabel}
\usetikzlibrary{matrix,arrows,decorations.pathmorphing}

\title{Maps on the Morse Boundary}
\author{Qing Liu}

\makeatletter
\numberwithin{equation}{section}
\numberwithin{figure}{section}

\theoremstyle{plain}
\newtheorem{thm}{Theorem}[section]
\newtheorem{lem}[thm]{Lemma}
\newtheorem{cor}[thm]{Corollary}
\newtheorem{prop}[thm]{Proposition}

\theoremstyle{definition}
\newtheorem{defn}[thm]{Definition}
\newtheorem{eg}[thm]{Example}

\theoremstyle{remark}
\newtheorem{rmk}[thm]{Remark}

\newcommand{\bbR}{{\mathbb{R}}}
\newcommand{\bbZ}{{\mathbb{Z}}}

\usepackage{babel}

\tikzset{node distance=1.5cm, auto}

\begin{document}
\maketitle

\begin{abstract}
For a proper geodesic metric space $X$, the Morse boundary $\partial_*X$ focuses on the hyperbolic-like directions in the space $X$. It is a quasi-isometry invariant. That is, a quasi-isometry between two hyperbolic spaces induces a homeomorphism on their boundaries. In this paper, we investigate additional structures on the Morse boundary $\partial_*X$ which determine $X$ up to a quasi-isometry. We prove that, for $X$ and $Y$ proper, cocompact spaces, a homeomorphism $f$ between their Morse boundaries is induced by a quasi-isometry if and only if $f$ and $f^{-1}$ are bihölder, or quasi-symmetric, or strongly quasi-conformal.

\end{abstract}
\section{Introduction}
The visual boundaries of hyperbolic metric spaces have been extensively studied. They play a critical role in the study of the geometry, topology and dynamics of hyperbolic spaces.
Gromov \cite{gromov1987hyperbolic} showed that they are quasi-isometry invariant. That is, a quasi-isometry between two hyperbolic metric spaces induces a homeomorphism between their boundaries. Hence the visual boundary is well-defined for a hyperbolic group. The boundary of a hyperbolic space is metrizable. If we fix visual metrics on the hyperbolic boundaries, the homeomorphism induced by a quasi-isometry between two hyperbolic spaces satisfies a variety of metric properties. It is bihölder, quasi-conformal, quasi-mobius and power quasi-symmetric. Quasi-mobius is a condition that bounds the distortion of cross-ratios and quasi-conformal bounds the distortion of metric spheres and annuli. These notions have been studied by Otal, Pansu, Tukia and Vaisala \cite{otal1992geometrie, pansu1989dimension, tukia1980two, tukia1982quasiconformal, tukia1984bilipschitz, vaisala2006lectures}. Quasi-symmetric maps have appeared \cite{tukia1980quasisymmetric, vaisala1981quasisymmetric}.

There is a natural question, to what extend the converse of this is true. That is, what extent is a hyperbolic space $X$ determined by its boundary $\partial X$. A result of F.Paulin \cite{paulin1996groupe} and M.Bourdon \cite{ledrappier1994structure} answers this question:
Let $G_1$ and $G_2$ be two word-hyperbolic groups. Suppose $h:\partial G_1\to \partial G_2$ is a homeomorphism such that $h$ and $h^{-1}$ are quasi-Mobius or quasi-conformal. Then $h$ extends to a quasi-isometry $f:G_1\to G_2$. In our paper, we will use Paulin's idea to approach the problem of extending a map between the Morse boundaries to a quasi-isometry between the spaces. Here the definition of quasi-conformal in the sense of Paulin is different from the one used by Tukia and others.

In the paper \cite{bonk2000embeddings}, M.Bonk and O.Schramm state a different result using power quasi-symmetric maps on the boundaries. They show that, if $f:(\partial X, d_{x_0, \epsilon_X})\to (\partial Y, d_{y_0, \epsilon_Y})$ is a power quasi-symmetry on the boundaries of two hyperbolic spaces $X, Y$, then $f$ extends to a quasi-isometry $h:X\to Y$. Their idea is different from Paulin.

Boundaries can be defined for more general metric spaces. We can define a visual boundary of CAT(0) spaces similarly. But this boundary is not a quasi-isometry invariant. An example of Croke and Kleiner \cite{croke2000spaces} shows that there is a group acting geometrically on two CAT(0) spaces with different topological boundaries. In \cite{CS2014}, Charney and Sultan constructed a new boundary for CAT(0) spaces which is defined by restricting to rays satisfying a contracting property. This boundary is originally called the contracting boundary and it is a quasi-isometry invariant. In CAT(0) spaces, the contracting property is equivalent to the Morse property. Later, Cordes \cite{Cordes} generalized this construction to arbitrary proper geodesic metric spaces using rays satisfying the Morse properties. These boundaries are called Morse boundaries. They are also quasi-isometry invariant. Thus the Morse boundary is well-defined for any finitely generated group. The aim of this boundary is to study hyperbolic behavior in a non-hyperbolic space. Several papers \cite{CS2014, C16, Cordes, cordes2016boundary, cordes2017stability, Murray, charney2019quasi, liu2019dynamics} studied Morse boundaries that many properties and applications of hyperbolic boundaries generalize to Morse boundaries of more general groups. 

Denote the Morse boundary of $X$ by $\partial_*X$.
In \cite{charney2019quasi}, Charney, Cordes and Murray investigate the question of when a homeomorphism of Morse boundaries is induced by a quasi-isometry of the interior spaces. They show the following result, which is an analogue of Paulin's theorem.
\begin{thm}[\cite{charney2019quasi}]
Let $h:\partial_*X\to \partial_*Y$ be a homeomorphism between the Morse boundaries of two proper, cocompact geodesic metric spaces. Assume that $\partial_*X$ contains at least three points. Then $f$ is induced by a quasi-isometry $h: X \to Y$ if and only if  $f$ and $f^{-1}$ are 2-stable and quasi-mobius.
 \end{thm}
 
 Sarah Mousley and Jacob Russel have proved an analogous result for Heirerarchically Hyperbolic groups \cite{mousley2019hierarchically}.

The purpose of this paper is to investigate additional structures on Morse boundaries $\partial_*X$ which determine $X$ up to a quasi-isometry. We show the following theorem.

\begin{thm}[\bf Main theorem]\label{main}
Let $X$ and $Y$ be proper, cocompact geodesic metric spaces and assume that $\partial_*X$ contains at least three points. Let $f: \partial_*X \to \partial_*Y$ be a homeomorphism. Then the following are equivalent.

\begin{enumerate}
\item $f$ is induced by a quasi-isometry $h: X \to Y$.
\item $f$ and $f^{-1}$ are bihölder.
\item $f$ and $f^{-1}$ are quasi-symmetric.
\item $f$ and $f^{-1}$ are strongly quasi-conformal. 
\end{enumerate} 
\end{thm}

Note that in general, the Morse boundary is neither metrizable nor compact, so it seems that there is no obvious way to talk about bihölder maps, quasi-symmetries and quasi-conformal maps on the Morse boundary. In the current paper, we provide an approach to this question, and define these notions on the Morse boundary. The starting step is that, in the paper \cite{cordes2017stability}, Cordes and Hume show that, for a proper geodesic metric space, the Morse boundary is the direct limit of a collection of Gromov boundaries. We know that each Gromov boundary is metrizable. This provides a possibility of generalizing notions defined in a metric space to the Morse boundary.

Combing the Main Theorem \ref{main} and the work of Charney, Cordes and Murray \cite{charney2019quasi}, we have the following theorem. We know that these four conditions in the main theorem  are equivalent to the 5th one:
$f$ and $f^{-1}$ are $2$-stable and quasi-möbius.

To define these notions, bihölder, quasi-symmetric, strongly quasi-conformal maps, we require a property called basetriangle stable. We will talk about these in section 3. The $2$-stable property is equivalent to the basetriangle stable property in our setting by Proposition \ref{2=bstri}. This means we have $2$-stable maps for free in the case of bihölder, quasi-symmetric, strongly quasi-conformal homeomorphisms under the assumption of Theorem \ref{main}.

Also we get the following corollary in the case of hyperbolic spaces. Please check the definitions in section 3. These notions between the Morse boundaries are different with that in the usual case of metric spaces.
\begin{cor}
Let $X$ and $Y$ be proper, cocompact geodesic, hyperbolic spaces. 
Suppose that $\partial X$ contains at least three points.
Let $f: \partial X \to \partial Y$ be a homeomorphism. Then the following are equivalent.

\begin{enumerate}
\item $f$ is induced by a quasi-isometry $h: X \to Y$.
\item $f$ and $f^{-1}$ are bihölder.
\item $f$ and $f^{-1}$ are quasisymmetric.
\item $f$ and $f^{-1}$ are strongly quasi-conformal. 
\item $f$ and $f^{-1}$ are quasi-mobius. 
\end{enumerate} 
\end{cor}

A paper of Cashen and Mackay \cite{cashen2019metrizable} introduces a metrizable topology on the Morse boundary. It is still interesting to know whether a full analogue of Paulin's theorem holds for this modified Morse boundary. Qing, Rafi and Tiozzo \cite{qing2019sublinearly} introduce a new boundary for CAT(0) spaces, called sublinearly Morse boundary, which is strictly larger than the Morse boundary. Qing and Zalloum \cite{qing2019rank} show that a homeomorphism $f: \partial_{K}G\to\partial_{K}G'$ is induced by a quasi-isometry if and only if $f$ is Morse quasi-mobius and stable, where $G$ and $G'$ are CAT(0) groups and $\partial_{K}G$ is the sublinearly Morse boundary of $G$. It is still interesting to know if there is a way to define quasi-conformal to get an analogue of our main theorem. 

The paper is organized as follows. In section 2, we review known properties of the Morse boundary and metrics on the Gromov boundary.  We define bihölder maps, quasi-symmetries and strongly quasi-conformal maps on the Morse boundary and prove one direction of the main theorem in section 3. In section 4, under any of these conditions, we extends a map between the boundary to the interior and show that it is a quasi-isometry.

{\bf Acknowledgment.} We thank Ruth Charney for helpful comments.

\section{Preliminaries}
Let $X$ be a metric space. We use $[x, y]$ to represent a geodesic between $x, y\in X$. We say a metric space $X$ is {\it proper} if any closed ball in $X$ is compact. If $A$ is a subset in $X$, the {\it r-neighborhood} of $A$ in $X$ is denoted by $\mathcal{N}_{r}(A)$. The {\it Hausdorff distance} $d_{\mathcal{H}}(A_1, A_2)$ between two subsets $A_1$ and $A_2$ is defined by $\inf\{ r\ |\ A_1\subset \mathcal{N}_{r}(A_2), A_2\subset \mathcal{N}_{r}(A_1)\}$.

\begin{defn}
Let $f: (X, d_X)\to (Y, d_Y)$ be a map between two metric spaces $X$ and $Y$, $K\ge 1, C\ge 0$ . If for all $x_1, x_2\in X$,
$$K^{-1}d_{X}(x_1, x_2)-C\le d_Y(f(x_1), f(x_2))\le Kd_X(x_1, x_2)+C,$$
then $f$ is called a {\it $(K, C)$-quasi-isometric embedding.}
If, in addition, $d_Y(f(x), Y)\le C$ for all $x\in X$, then $f$ is called a {\it $(K, C)$-quasi-isometry}.
If $X$ is an interval of $\bbR$, then a $(K, C)$-quasi-isometric embedding $f$ is called a {\it $(K, C)$-quasi-geodesic}. We use the image of $f$ to describe the quasi-geodesic.
\end{defn}

\begin{defn}
Let $N$ be a function from $[1, \infty)\times [0, \infty)$ to $[0, \infty)$. A geodesic $\gamma$ in $X$ is {\it $N$-Morse} if for any $(K, C)$-quasi-geodesic $\alpha$ with endpoints on $\gamma$, we have $\alpha\subset \mathcal{N}_{N(K, C)}(\gamma)$. The function $N$ is called a Morse gauge.
\end{defn}

We know that in a hyperbolic space $X$, there exists a Morse gauge $N$ such that all geodesics in $X$ are $N$-Morse. If we consider the Euclidean spaces, there is no infinite Morse geodesic. In more general metric spaces, some infinite geodesics are Morse, others are not. Morse geodesics in a proper geodesic metric space behave similarly to geodesics in a hyperbolic metric space. 

\subsection{The Morse boundary}
Now let us define the Morse boundary. All the details can be found in \cite{Cordes}.

Let $X$ be a proper geodesic metric space and $x_0$ be a basepoint. The {\it Morse boundary} $\partial_*X$ of $X$ is the set of equivalence classes of all Morse geodesic rays with basepoint $x_0$ and we say two Morse geodesics rays are {\it equivalent} if they have finite Hausdorff distance. Fix a $N$ Morse gauge, we topologies the set 
$$\partial_*^{N}X_{x_0}=\{[\alpha]\ | \mbox{ there exists an }  N\mbox{-Morse geodesic ray } \beta\in [\alpha] \mbox{ with basepoint } x_0 \}$$ with the compact-open topology. 

Consider the set $\mathcal{M}$, all Morse gauges. We put a partial ordering on $\mathcal{M}$: $N\le N'$ if and only if $N(K, C)\le N'(K, C)$ for all $K, C$. The Morse boundary is defined to be 
$$\partial_*X_{x_0}=\varinjlim_{\mathcal{M}}\partial_*^{N}X_{x_0}$$
with the induced direct limit topology, i.e., a set $V$ is open in $\partial_*X_{x_0}$ if and only if $V\cap \partial_*^NX_{x_0}$ is open in $\partial_*^NX_{x_0}$ for all $N$ Morse gauges.

Another equivalent way to give the compact-open topology on $\partial_*^NX_{x_0}$ is using a system of neighborhoods. 
For a proper geodesic space $X$, the Morse boundary $\partial_*X_{x_0}$ is basepoint independent, we will omit the basepoint from the notation, the Morse boundary is denoted by $\partial_*X$. 

Here let us list some properties of Morse triangles and Morse geodesics. 

\begin{lem}[\cite{liu2019dynamics}]\label{unif M}
For any $N$ Morse gauge, there exists $N'$ such that any segment of an $N$-Morse geodesic is $N'$-Morse.
\end{lem}

The next lemma says that a geodesic which is close to a Morse geodesic is uniformly Morse. It is an easy exercise we leave to the reader.

\begin{lem}\label{close M}
Let $\alpha: [a, b]\to X$ and $\beta$ be geodesics in a geodesic metric space $X$. Suppose that $\beta$ is $N$-Morse, and $d(\alpha(a), \beta)$, $d(\alpha(b), \beta)\le \epsilon$. Then $\alpha$ is $N'$-Morse where $N'$ depends only on $N$ and $\epsilon$.

\end{lem}

Recall that a geodesic triangle $T$ is called $\delta$-slim if $\delta$-neighborhood of any two sides of $T$ covers the third side. 

\begin{lem}[Slim Triangles,\cite{charney2019quasi}]\label{slim} 
Let $X$ be a proper geodesic metric space. Let $a, b, c$ be points in $X\cup \partial_*X$. Suppose that two sides of the triangle $T(a, b, c)$ are $N$-Morse, then the third side is $N'$-Morse and the triangle is $\delta_N$-slim, where $N'$ and $\delta_N$ depend only on $N$.
\end{lem}

\subsection{Visual metrics and Topology on hyperbolic boundaries}

In this subsection, we will give a quick review about the construction of sequential boundary and visual metrics for a hyperbolic space. All of this can be found in \cite{ghys1990groupes} or \cite[Chpater III.H]{bridson}. 
 
 Let us give the definition of hyperbolic space in the sense of Gromov.
 \begin{defn}
 Let $X$ be a metric space and let $x, y, z\in X$. The {\it Gromov product} of $x$ and $y$ with respect to $z$ is defined by 
 $$(x\cdot y)_{z}=\frac{1}{2}(d(z, x)+d(z, y)-d(x, y)).$$
 \end{defn}

 \begin{defn}\label{Ghyp}
 Let $X$ be a metric space and let $\delta\ge 0$ be a constant. We call $X$ is $\delta$-hyperbolic if for all $w, x, y, z\in X$, we have $$(x\cdot y)_w\ge\min\{(x\cdot z)_w, (z\cdot y)_w\}-\delta.$$
 \end{defn}

Let $X$ be a $\delta$-hyperbolic space. 
Let $p\in X$ be a basepoint and $(x_n)$ be a sequence in $X$. We say the sequence $(x_n)$ {\it converges at infinity} if $(x_i\cdot x_j)_{p}\to \infty$ as $i, j\to \infty$. Two convergent sequences $(x_n)$, $(y_m)$ are called {\it equivalent} if $(x_i\cdot y_j)_{p}\to \infty$ as $i, j\to \infty$. We use $\lim x_n$ to denote the equivalence class of $(x_n)$.

The {\it sequential boundary} of $X$, $\partial_sX$ is defined to be the set of equivalence classes of convergent sequences. In our paper, we omit the notation $s$. For a hyperbolic space $X$, $\partial X$ is denoted as its boundary.
We extend the Gromov product to $\partial X$ by:
$$(x\cdot y)_{p}=\sup \underset{n, m\to \infty}{\lim\inf}(x_n\cdot y_m)_{p},$$
where the supremum is taken over all sequences $(x_n),$ $(y_m)\in X$ such that $x=\lim x_n, y=\lim y_m$. 

We have the following properties for the Gromov product.

\begin{lem}\label{GP}
Let $X$ be a $\delta$-hyperbolic space and $p\in X$ be a basepoint. Then:

\begin{enumerate}
\item

$(x\cdot y)_p=\infty$ if and only if $x=y\in \partial X$.

\item

For all $x, y\in \partial X$ and all sequences $(x_n), (y_m)\in X$ with $x=\lim x_n, y=\lim y_m$,
we have $$(x\cdot y)_p-2\delta\le \underset{n, m\to \infty}{\lim\inf}(x_n\cdot y_m)_p\le (x\cdot y)_p.$$

\item
For any $x, y, z\in \partial X\cup X$,  we have $(x\cdot y)_{p}\ge \min\{(x\cdot z)_p, (y\cdot z)_p\}-2\delta.$
\end{enumerate}

\end{lem}

\begin{defn}
For a hyperbolic space $X$ with basepoint $p$, we say a metric $d$ on $\partial X$ is a {\it visual metric} with parameter $\epsilon$ if there exist constants $k_1, k_2>0$ so that $$k_1 e^{-\epsilon(x\cdot y)_{p}}\le d(x,y)\le k_2 e^{-\epsilon(x\cdot y)_{p}}$$ for all $x, y\in \partial{X}$.
\end{defn}

Here is the standard construction of the metrics on $\partial X$. If $x, y\in \partial X, p\in X, \epsilon>0$, 
let $$d_{\partial X, p, \epsilon}(x, y)=d_{p, \epsilon}(x, y)=\inf \left\{\sum_{i=1}^{n}e^{-\epsilon(x_{i-1}\cdot x_{i})_{p}}\right\},$$
where the infimum is taken over all finite sequence $x=x_1, x_2,..., x_n=y$ in $\partial X$, no bound on $n$.

\begin{thm}\cite[Section 7.3]{ghys1990groupes}\label{metric}
Let $X$ be a $\delta$-hyperbolic space. Let $\epsilon$ be a positive constant so that $e^{2\delta\epsilon}\le \sqrt{2}$. Then $\displaystyle{ (3-2e^{2\delta\epsilon})e^{-\epsilon(x\cdot y)_{p}}\le d_{p, \epsilon}(x, y)\le e^{-\epsilon(x\cdot y)_{p}}}$ for any $x, y\in \partial X$
\end{thm}

\begin{rmk}\label{B-eq}
The canonical gauge $\mathcal{G}(X)$ on $\partial X$ is the set of all metrics of the form $d=d_{p, \epsilon}$. Two metrics $(\partial X, d_{p_1, \epsilon_1})$ and $(\partial X, d_{p_2, \epsilon_2})$ are {\it B-equivalent} if there exists a constant $k>0$ such that 
$$k^{-1}d_{p_2, \epsilon_2}^{\epsilon_1}\le d_{p_1, \epsilon_1}^{\epsilon_2}\le kd_{p_2, \epsilon_2}^{\epsilon_1}.$$

The theorem says that if $p_1, p_2\in X$ and $\epsilon_1, \epsilon_2>0$ such that $e^{2\delta\epsilon_1}, e^{2\delta\epsilon_2}\le\sqrt{2}$, then $(\partial X, d_{p_1, \epsilon_1})$ and $(\partial X, d_{p_2, \epsilon_2})$ are B-equivalent. 
A metric $d_{p, \epsilon}$ in $\mathcal{G}(X)$ induces a topology on $\partial X$ and this topology does not depend on the choice of the metric $d_{p, \epsilon}$ in $\mathcal{G}(X)$.
\end{rmk}

\subsection{Metric Morse boundaries}
With the Gromov product in hand, we can define metric Morse boundaries, this is the work of Cordes and Hume in their paper \cite{cordes2017stability}. I will give a quick review of their work. 

Let $(X, d)$ be a geodesic metric space. For a Morse gauge $N$, let $$X_{x_0}^{(N)}=\{ x\in X\ | \mbox{ all geodesics between } x \mbox{ and } x_0 \mbox{ are } N\mbox{-Morse } \}.$$ This metric space is not necessarily a geodesic space. However, it is a $32N(3, 0)$-hyperbolic spaces in the sense of Definition \ref{Ghyp}. 



Thus we get a sequential boundary $\partial X_{x_0}^{(N)}$ with a visual metric $d_{x_0,\epsilon_{N}}\in \mathcal{G}(X_{x_0}^{(N)})$ for some visibility parameter $\epsilon_N$. In our paper, it is convenient to fix a visibility parameter $\epsilon_N$ for each $N$ and to work with this visual metric $(\partial X_{x_0}^{(N)}, d_{x_0, \epsilon_N})$.

\begin{defn}
Let $x, y\in \partial X_{x_0}^{(N)}$. The {\it $N$-Gromov product} of $x$ and $y$ is defined by 
$$(x\cdot_N y)_{x_0}=\sup \underset{n, m\to \infty}{\lim\inf}(x_n\cdot y_m)_{x_0},$$
where the supremum is taken over all sequences $(x_n),$ $(y_m)$ in $X_{x_0}^{(N)}$ such that $x=\lim x_n, y=\lim y_m$.

The Gromov products of $x, y\in \partial X_{x_0}^{(N)}$ depends on $N$. From Lemma 3.11 in \cite{cordes2017stability}, if $x, y\in \partial X_{x_0}^{(N)}$ and $N\le N'$, then 
$$(x\cdot_N y)_{x_0}\le(x\cdot_{N'} y)_{x_0}\le (x\cdot_N y)_{x_0}+64N'(3, 0).$$

\end{defn}

\subsection {Relationship between $\partial X_{x_0}^{(N)}$ and $\partial_*^{N} X_{x_0}$}

Cordes and Hume \cite{cordes2017stability} gave the definition of the space $X_{x_0}^{(N)}$ and proved many properties. In this subsection, let us review these properties since we will use them in later sections. Also we need to be more careful about the space $X_{x_0}^{(N)}$ since it is not necessarily geodesic.

Now assume that $(X, d)$ is a proper geodesic metric space. We will discuss the relationship between $\partial X_{x_0}^{(N)}$ and $\partial_*^{N}X_{x_0}$. 
Note that $\partial X_{x_0}^{(N)}$ is the sequential boundary with respect to the Gromov product and $\partial_*^{N}X_{x_0}$ is the $N$-Morse boundary with the compact-open topology.

The Morse gauge of segments of an $N$-Morse geodesic might be different from $N$, so the natural map from $\partial_*^N{X_{x_0}}$ to $\partial X_{x_0}^{(N)}$ does not always exist, see figure \ref{example}. 
This problem is not crucial, we will use $\partial X_{x_0}^{(N')}$ instead of $\partial X_{x_0}^{(N)}$ for some Morse gauge $N'$.
Thus we have the following theorem.

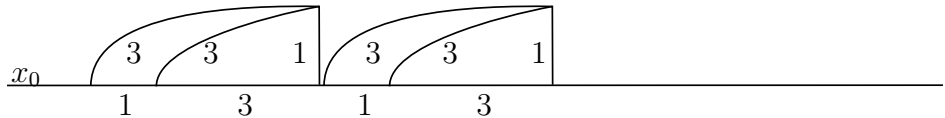
\begin{figure}[!ht]
\labellist

\pinlabel $x_0$ at 17 11
\pinlabel $1$ at 177 25
\pinlabel $1$ at 75 -5
\pinlabel $3$ at 80 25
\pinlabel $3$ at 125 25
\pinlabel $3$ at 145 -5

\pinlabel $1$ at 317 25
\pinlabel $1$ at 215 -5
\pinlabel $3$ at 220 25
\pinlabel $3$ at 265 25
\pinlabel $3$ at 285 -5

\endlabellist
\includegraphics[width=5.in]{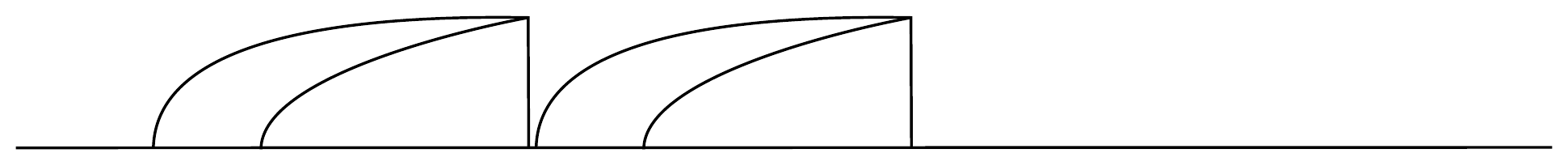}
\caption{The horizontal line is 
$\bbR_{\ge 0}$. Let $x_0=0$.
At each triple (t, t+1, t+4) of points, a $(3,3,1)$-type graph is attached. This horizontal line is a $2$-Morse geodesic ray. But $\partial X_{x_0}^{(2)}$ is empty.}
\label{example}
\end{figure}

\begin{thm}\label{s and *}

For any $N$ Morse gauge, there exists $N'$ depending only on $N$ such that the following holds.

Let $X$ be a proper geodesic metric space. Let $x_0$ be a basepoint.
There exist two natural embeddings 
$$i_N: \partial X_{x_0}^{(N)}\to \partial_*^{N}X_{x_0}, \mbox{ and }$$  
$$j_{N, N'}: \partial_*^{N}X_{x_0} \to \partial X_{x_0}^{(N')}.$$
\end{thm}

\begin{proof}
From Lemma \ref{unif M}, any segments of $N$-Morse geodesics are $N'$-Morse, where $N'$ depends only on $N$. We follow the proof of Theorem 3.14 in \cite{cordes2017stability}. They constructed two maps $i_{N}$ and $j_{N, N'}$ and showed that $i_N$ is well-defined. It is not hard to show that these two maps are injections. Following the Remark 3.17 in \cite[III.H.]{bridson}, when $X$ is proper, these injections $i_N$ and $j_{N, N'}$ are embeddings.
\end{proof}

\begin{rmk}\label{2 tops}

From the above theorem, we conclude that there exists a homeomorphism between $\partial_*X_{x_0}$ and $\varinjlim_{\mathcal{M}}\partial X_{x_0}^{(N)}$.
Lemma 5.4 in \cite{liu2019dynamics} tells us that the subspace topology on $\partial_*^{N}X_{x_0}$ induced by the inclusion map from $\partial_*^NX_{x_0}$ to $\partial_*X_{x_0}$ is the compact-open topology of $\partial_*^NX_{x_0}$. 
Together with Theorem \ref{s and *}, when $X$ is a proper geodesic metric space, the subspace topology of $\partial X_{x_0}^{(N)}$, which is induced by the inclusion map from $\partial X_{x_0}^{(N)}$ to $\partial_* X_{x_0}$, agrees with the topology of $\partial X_{x_0}^{(N)}$ induced by a visual metric $d_{x_0, \epsilon_N}\in \mathcal{G}(X_{x_0}^{(N)})$.
\end{rmk}

In the proof of Theorem \ref{thm 2}, we will consider convergence sequences in different topological spaces. We need the following lemma which follows from Lemma 5.3 in \cite{liu2019dynamics} and Theorem \ref{s and *}.

\begin{lem}\label{conv} Let $x_0$ be a basepoint in a proper geodesic metric space $X$. Let $p_n$ and $p$ be points in $\partial_*X$. Then the following are equivalent: 
\begin{enumerate}
\item The sequence $p_n$ converges to $p$ in the topology of $\partial_*X_{x_0}$.
\item There exists $N$ such that $p_n, p\in \partial_*^{N}X_{x_0}$ and $p_n$ converges to $p$ in the topology of $\partial_*^{N}X_{x_0}$. 
\item There exists $N'$ such that $p_n, p\in \partial X_{x_0}^{(N')}$ and $p_n$ converges to $p$ in the topology of $\partial X_{x_0}^{(N')}$. 
\item There exists $N'$ such that $p_n, p\in \partial X_{x_0}^{(N')}$ and $d_{x_0, \epsilon_{N'}}(p_n, p)\to 0$. 

\end{enumerate}
\end{lem}

With Theorem \ref{s and *}, it is not hard to show the following. In the case of a proper geodesic hyperbolic space, we have this \cite[III.H. 3.18(3)]{bridson}.
\begin{lem}\label{dist-GP}
For any Morse gauge $N$, there exists a constant $C_N>0$ such that the following holds.

Let $X$ be a proper geodesic metric space and $x_0\in X$ be a basepoint. For any $p, q\in \partial X_{x_0}^{(N)}$, let $\gamma$ be any geodesic between $p$ and $q$.
Then we have $|(p._{N} q)_{x_0}-d_X(x_0, \gamma)|\le C_N.$


\end{lem}

\section{Maps on the Morse boundary}

For proper metric spaces $X, Y$, we would like to study bihölder maps, quasi-symmetries and strongly quasi-conformal maps between their Morse boundaries $\partial_*X$ and $\partial_*Y$. 
Normally, one defines these three types of maps for metric spaces. The fact that Morse boundaries are not metrizable is a problem. 
But instead, we have a collection of metric Morse boundaries if we choose a basepoint. That is, for any $N$ Morse gauge, basepoint $x_0\in X$, the boundary $\partial X^{(N)}_{x_0}$ has a visual metric. 
Here, however, we will consider homeomorphisms between Morse boundaries that, a priori, have nothing to do with the interior points. The following result makes a connection between ideal triangles and interior points. It comes from the work of Charney, Cordes, Murray. Using ideal triangles on the boundary rather than basepoints in the interior is more reasonable in our setting.

Denote by $\partial_*X^{(n, N)}$, the set of $n$-triples of distinct points $(p_1,..., p_n)$ in $\partial_*X$ such that any bi-infinite geodesic from $p_i$ to $p_j$ is $N$-Morse. For any ideal triangle $T(a, b ,c)$, where $(a, b, c)\in \partial_*X^{(3, N)}$, we can define centers of $T(a, b, c)$. The center is not unique, but the next lemma shows that they are a bounded set.
\subsection{Triangles and coarse center}
\begin{lem}\cite[Lemma 2.5]{charney2019quasi}\label{coarse}

Let $X$ be a proper geodesic metric space. Let $(a, b, c)\in \partial_{*}X^{(3, N)}$.
\\Set $E_k(a, b, c)=\{x\in X\ | \ x \mbox{ lies within } k \mbox{ of all three sides of some triangle } T(a, b, c)\}$. For any $k\ge \delta_{N}$, we have the following:

\begin{enumerate}
\item $E_k(a, b, c)$ is non-empty,
\item $E_k(a, b, c)$ has bounded diameter L depending only on $N$ and $k$.
\end{enumerate}
\end{lem}

Set $K(a, b, c)=1+\inf\{k\ | \ E_k(a, b, c)\neq \emptyset\}$. It depends only on the vertices, not on $N$. From the above lemma we know that $K(a, b, c)$ is bounded by $1+\delta_{N}$. 

Any point in $E_{K(a, b, c)}(a, b, c)$ is called a {\it coarse center} of $(a, b, c)$. For simplicity, we will suppress the notation $E_{K(a, b, c)}(a, b, c)$ as $E(a, b, c)$.
 
The following two observations are not hard, but they are useful in the later sections. 

\begin{lem}\label{cen bspt}
For any Morse gauge $N$, there exists $N'$ such that the following holds.

Let $X$ be a proper geodesic metric space. For any $(a, b, c)\in \partial_*X^{(3, N)}$ and coarse center $x_0\in E(a, b, c)$, we have $a, b, c \in \partial X_{x_0}^{(N')}$.
\end{lem}

\begin{proof}
Let $x_0$ be a coarse center of $(a, b, c)$. By Lemma \ref{coarse}, there exists a $N$-Morse geodesic $\gamma$ from $a$ to $b$ and a point $x_1\in \gamma$  such that $d(x_0, x_1)\le 1+\delta_N$. By Lemma \ref{unif M}, the geodesic $[x_1, a)$ is $N_0$-Morse where $N_0$ depends only on $N$.
Now Lemma \ref{close M} implies that any geodesic from $x_0$ to $a$ is $N_1$-Morse for some Morse gauge $N_1$ depending only on $N_0$ and $1+\delta_N$. It means that $a\in \partial_*^{N_1}X_{x_0}$. An analogous argument proves that $b, c\in \partial_*^{N_1}X_{x_0}$. By Theorem \ref{s and *}, there exists $N'$ depending only on $N_1$ such that $a, b, c\in \partial X_{x_0}^{(N')}$.
\end{proof}

The following lemma says that if two ideal triangles are sufficiently close to each other, then their coarse center sets are uniformly bounded.
\begin{lem}\label{close tri}
Let $X$ be a proper geodesic metric space and $x_0\in X$ be a basepoint. Consider the visual metric on the $N$-Morse boundary: $(\partial X_{x_0}^{(N)}, d_{x_0, \epsilon_{N}})$. 
Choose three distinct points $a, b, c\in \partial X_{x_0}^{(N)}$. 
There exists constants $\lambda$ and $D=D(N)$ such that the following holds.

Let  $a', b', c'\in \partial X_{x_0}^{(N)}$ be three distinct points which satisfies $$d_{x_0, \epsilon_{N}}(a, a'), d_{x_0, \epsilon_{N}}(b, b'), d_{x_0, \epsilon_{N}}(c, c')\le \lambda.$$ For any $p\in E(a, b, c), p'\in E(a', b', c')$, we have $$d_X(p, p')\le D.$$

\end{lem}

\begin{proof}
By the Lemma \ref{slim}, it is not hard to see that any ideal triangle with vertices in $\partial X_{x_0}^{(N)}$ is an $N'$-Morse triangle, 
hence it is $\delta$-slim, where $N'$ and $\delta$ depend only on $N$.
Choose a point $p_0\in E(a, b, c)$. There exists a bi-infinite geodesic $\alpha$ from $a$ to $b$ such that $d_X(p_0, \alpha)\le 1+\delta$.
That is, we have a point $p_1\in \alpha$ with $d_X(p_0, p_1)\le 1+\delta$.
We can choose $\lambda$ sufficiently small so that the distances from $p_0$ to any geodesic $(a, a')$ and to any geodesic $(b, b')$ are larger than $\delta+1$. 
See Figure \ref {close tri}.
Since all triangles with vertices $(a, b', a')$ and $(a, b, b')$ are $\delta$-slim, for any geodesic $(a, b')$ and $(a', b')$, there exist points $p_2\in (b', a)$ and $p_3\in (a', b')$ such that $$d_X(p_1,p_2), d_X(p_2, p_3)\le \delta.$$
It follows that $d(p_0, p_3)\le 1+3\delta$, where $p_3$ is a point in a geodesic from $a'$ to $b'$. A similar argument shows that, the point $p_0\in E_{1+3\delta}(a', b', c')$. From Lemma \ref{coarse}, it is easy to find constants $L_1, L_2$ such that $d_X(p, p_0)\le L_1$ and $d_X(p_0, p')\le L_2$ for any $p\in E(a, b, c), p'\in E(a', b', c')$. The constants $L_1$ and $L_2$ depend only on $1+3\delta$ and $N'$. Taking $D=L_1+L_2$, this shows the lemma.

\end{proof}

\begin{figure}[!ht]
\labellist

\pinlabel $p_0$ at 140 110
\pinlabel $p_3$ at 72 120
\pinlabel $p_1$ at 114 126
\pinlabel $p_2$ at 95 100
\pinlabel $a$ at 55 15
\pinlabel $a'$ at 38 28
\pinlabel $b$ at 90 227
\pinlabel $b'$ at 70 220
\pinlabel $c$ at 230 102
\pinlabel $c'$ at 232 115

\pinlabel $\partial X_{x_0}^{(N)}$ at -15 130 
\pinlabel $X$ at 140 30

\endlabellist
\includegraphics[width=3.in]{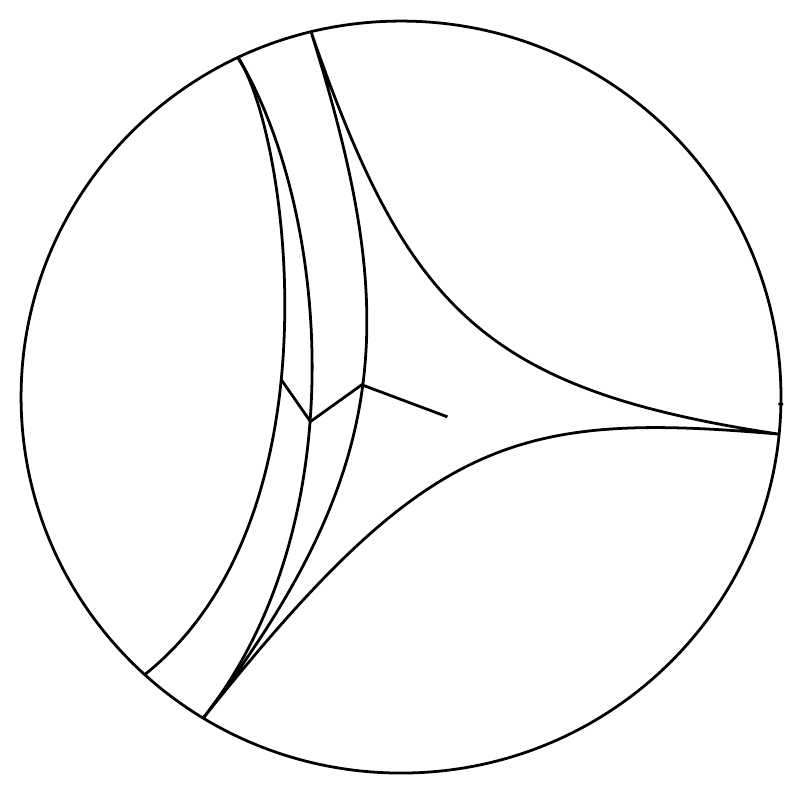}
\caption{Point $p_0$ is a coarse center of the triangle $T(a, b, c)$. Points $a'$, $b'$ and $c'$ are close to $a$, $b$ and $c$, respectively.} 
\label{close tri}
\end{figure}
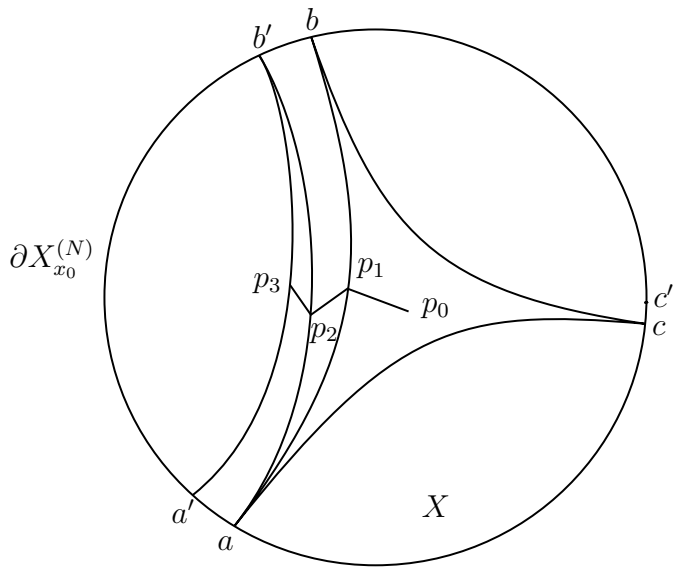

\subsection{$2$-stable maps and basetriangle stable maps}

Let $X$ be a proper geodesic metric space.
Let $\partial_{*}X^{(2,N)}=\{(p, q)\ |\ p, q\in\partial_*X \mbox{ and any geodesics between } p \mbox{ and } \\
q \mbox{ are } N\mbox{-Morse} \}$.
We say a map $f: \partial_*X \to \partial_*Y$ is {\it 2-stable} if for any $N$ there exists $N'$ such that $f(\partial_{*}X^{(2, N)})\subset \partial_{*}Y^{(2, N')}$.

From the Theorem 3.17 in \cite{charney2019quasi}, every quasi-isometry between two proper geodesic metric spaces induces a $2$-stable homeomorphism between their Morse boundaries. 
This homeomorphism does not depend on a choice of basepoint.
In our setting, we focus on the maps on the Morse boundary. It does not make sense to use basepoint in the interior since the map is not defined on the interior. However, using basetriangles on the Morse boundary seems to be reasonable. Indeed, every quasi-isometry indues a {\it basetriangle stable} map between the Morse boundaries in the following sense. Later we will discuss the relationship between $2$-stable maps and basetriangle stable maps.

\begin{defn}[Basetriangle stable maps]\label{tri sta}Let $X$ and $Y$ be proper geodesic metric spaces. An embedding $f: \partial_{*}X\hookrightarrow \partial_{*}Y$ is  $N_0${\it -basetriangle stable} if for any $N$, there exists a Morse gauge $N'$ such that 
\begin{itemize}
\item for all 3-triples $(a, b, c) \in \partial_{*}X^{(3, N_0)}$, and 
\item for all $x_0\in E(a, b, c), y_0\in E(f(a),f(b),f(c)),$ 
\end{itemize}
we have $$f(\partial {X_{x_0}^{(N)}})\subset\partial {Y_{y_0}^{(N')}}.$$
We say that $f$ is {\it basetriangle stable} if it is $N_0$-basetriangle stable for every $N_0$. \end{defn}

\begin{rmk}
Under the assumptions in Definition \ref{tri sta}, for simplicity, we usually suppress the notation $f|_{\partial X_{x_0}^{(N)}}: (\partial {X_{x_0}^{(N)}}, d_{x_0, \epsilon_N})\hookrightarrow \partial ({Y_{y_0}^{(N')}}, d_{y_0, \epsilon_{N'}})$ as $f_N: \partial X^{(N)}\hookrightarrow \partial Y^{(N')}$. However, for different basetriangles and coarse centers, the embedding $f_N$ is different.
\end{rmk}

This definition of basetriangle stable map is more complicated than the notion of 2-stable maps. But the motivation comes naturally from the quasi-isometry between two proper geodesic metric spaces.
When dealing with quasi-mobius maps, one focuses on bi-infinite Morse geodesics, so it is convenient to use $2$-stable maps. In our setting, we want to study other types of maps, like quasisymmetric and quasi-conformal maps between the Morse boundaries which are not metrizable in general. 
Since for any $N$, the map $f_{N}$ is between two metric spaces $\partial X_{x_0}^{(N)}$ and $\partial Y_{y_0}^{(N')}$, the basetriangle stable condition provides a way to define these properties for maps on the Morse boundaries.
The next proposition says that, basetriangle stable maps are usually weaker than $2$-stable maps.  However, under certain conditions they are equivalent. In the following sections, we will switch between these notions. 

\begin{prop}\label{2=bstri}
Let $X$ and $Y$ be two proper geodesic metric spaces.
If $f: \partial_*X\to \partial_*Y$ is a $2$-stable embedding, then $f$ is basetriangle stable.
Conversely, if $X$ is cocompact and $f$ is $N_0$-basetriangle stable for some $N_0$, then $f$ is $2$-stable.

In particular, if $h: X\to Y$ is a quasi isometry, then the induced homeomorphism between their Morse boundaries is basetriangle stable.

\end{prop}

\begin{proof}
First, we assume that $f:\partial_*X\to \partial_*Y$ is a $2$-stable embedding. Given any $N_0$, there exists $N_1$ depending on $f$ and $N_0$ so that $f(\partial_*X^{(3, N_0)})\subset\partial_*Y^{(3, N_1)}$.  For any $N$, we are going to find $N'$ such that the following holds.
For any $(a, b, c)\in\partial_*X^{(3, N_0)}, x_0\in E(a, b, c) \mbox{ and  } y_0\in E(f(a), f(b), f(c))$, we have
$$f(\partial X_{x_0}^{(N)})\subset \partial Y_{y_0}^{(N')}.$$
By Lemma \ref{cen bspt}, the geodesics 
$$[x_0, a) \mbox{ is } N_0' \mbox{-Morse and } [y_0, f(a))  \mbox{ is } N_1'\mbox{-Morse}.$$

\begin{figure}[!ht]
\labellist

\pinlabel $f$ at 101 42
\pinlabel $x_0$ at 40 35
\pinlabel $y_0$ at  165 35
\pinlabel $a$ at 16 9
\pinlabel $f(a)$ at 141 7
\pinlabel $b$ at 74 35
\pinlabel $f(b)$ at 203 35
\pinlabel $c$ at 25 72
\pinlabel $f(c)$ at 150 73
\pinlabel $p$ at 53 72
\pinlabel $f(p)$ at 180 73
\pinlabel $\partial_*Y$ at 165 0
\pinlabel $\partial_*X$ at 40 0
\pinlabel $X$ at 50 16
\pinlabel $Y$ at 175 17

\endlabellist
\includegraphics[width=6in]{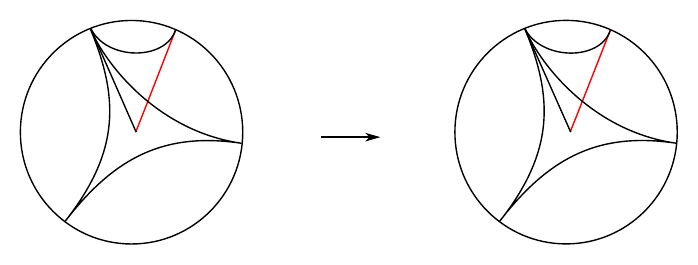}
\caption{$2$-stable embedding is basetriangle stable.} 
\label{stable 1}
\end{figure}
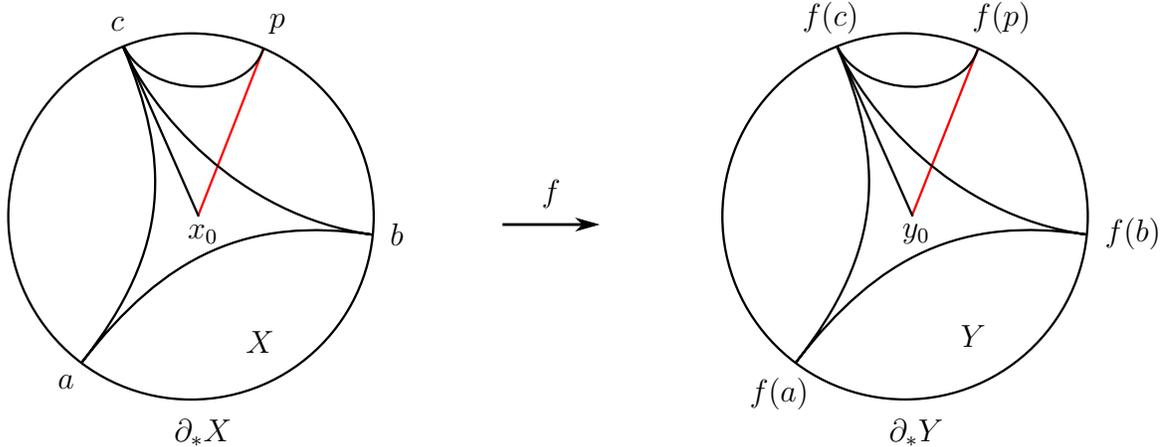

Let $p\in \partial X_{x_0}^{(N)}$. 
Consider triangles $T(x_0, c, p)$ and $T(y_0, f(c), f(p))$ as in Figure \ref{stable 1}. The geodesic $(c, p)$ is $N_2$-Morse by Lemma \ref{slim}. Since $f$ is $2$-stable, the geodesic $(f(c), f(p))$ is $N_3$-Morse. Lemma \ref{slim} implies that the geodesic $[y_0, f(p))$ is $N'$-Morse. Hence, $f(p)\in \partial Y_{y_0}^{(N')}$. Note that all Morse gauges $N_1, N_0', N_1', N_2$ and $N_3$ depend on $f$, $N$ and $N_0$ only. So the same holds for $N'$. Thus $f$ is $N_0$-basetriangle stable for every $N_0$.

Now assume that $X$ is cocompact, and $f$ is $N_0$-basetriangle stable. Fix a basetriangle $(a, b, c)\in\partial_*X^{(3, N_0)}$, $x_0\in E(a, b, c)$. For any $N$ and $(p, q)\in\partial_*X^{(2, N)}$. 
It is enough to show that the geodesic between $f(p)$ and $f(q)$ is $N'$-Morse where $N'$ depends only on $N$ and $f$. 
As in Figure \ref{stable 2},
choose a point $o$ in the geodesic $(p, q)$. Lemma \ref{unif M} tells us that $[o, p)$ and $[o, q)$ are $N'$-Morse where $N'$ depends on $N$.
Since $X$ is cocompact, there exists a group $G$ acting on $X$ isometrically and there exist $R>0$ and $g\in G$ such that $d_X(g(x_0), o)\le R.$

\begin{figure}[!ht]
\labellist

\pinlabel $p$ at 89 272
\pinlabel $q$ at 277 155
\pinlabel $a$ at 17 67
\pinlabel $g(a)$ at -10 141
\pinlabel $b$ at 48 32
\pinlabel $g(b)$ at 200 275
\pinlabel $g(c)$ at 175 1
\pinlabel $c$ at 88 7
\pinlabel $g(x_0)$ at 130 130
\pinlabel $o$ at 150 181
\pinlabel $\partial_*X$ at 250 35
\pinlabel $X$ at 220 70

\endlabellist
\includegraphics[width=2.7in]{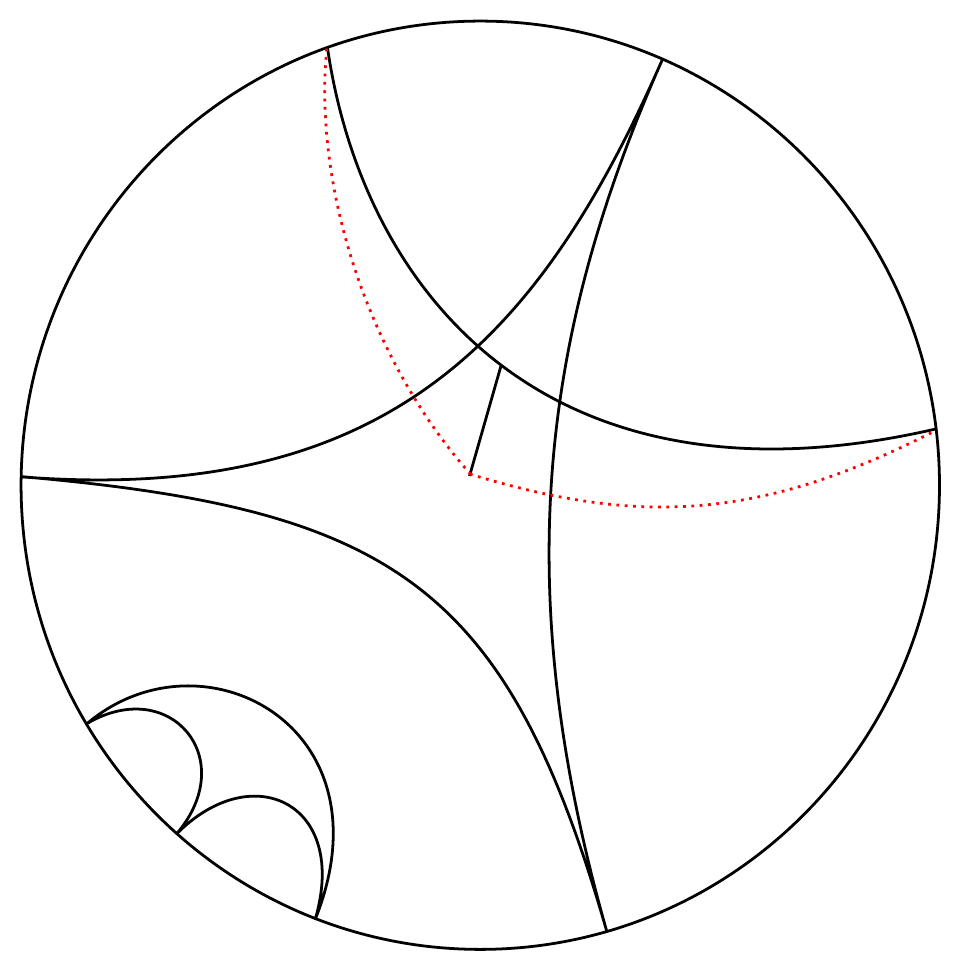}
\caption{Geodesic $(p, q)$ is $N$-Morse. The distance $d(o, g(x_0))\le R$. } 
\label{stable 2}
\end{figure}
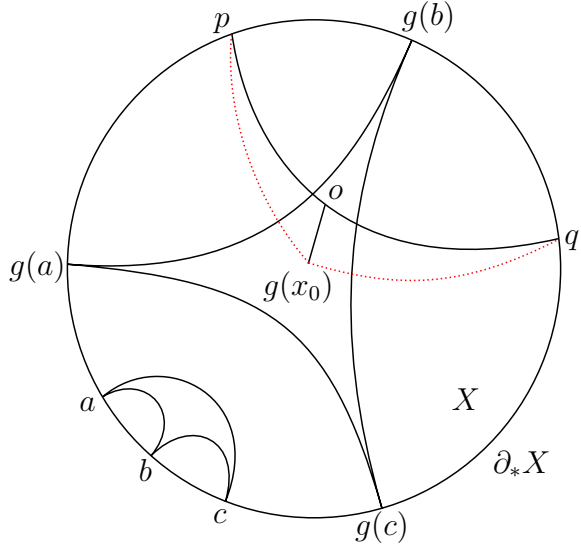

By Lemma \ref{close M} we have the geodesics $[g(x_0), p)$ and $[g(x_0), q)$ are $N_1$-Morse, where $N_1$ depends on $N'$ and $R$. That is,  
$$p, q\in \partial X_{g(x_0)}^{N_1} \mbox{ and } g(x_0)\in E(g(a), g(b), g(c)).$$
By Definition \ref{tri sta}, 
there exists Morse gauge $N_1'$ depending on $f$ and $N_1$ such that for any point $y_0\in E(f(g(a)), f(g(b)), f(g(c)))$, we have 
$$f(p), f(q)\in \partial Y_{y_0}^{N_1'}.$$ 

Now consider the triangle $T(y_0, f(p), f(q))$.
By Lemma \ref{slim} again, there exists $N'$ depending on $N_1'$ so that $(f(p), f(q))\in\partial_*Y^{(2, N')}$. The Morse gauge $N'$ depends only on $N$, $R$ and $f$. Thus $f$ is $2$-stable.
 
\end{proof}

The condition cocompactness is important in the converse part.
An $N_0$-basetriangle stable homeomorphism $f$ does not guarantee that $f$ is $2$-stable. See the following examples.

\begin{eg}
Let $X_0$ be a Euclidean plane $\bbR^2$. 
Choose two points $(\epsilon, 0)$ and $(-\epsilon, 0)$ for some constant $\epsilon$.
At points $(m, 0)\in \bbZ^{2}\subset \bbR^{2}$, a vertical ray $r_m$ is attached. At the point $(\epsilon, 0)$ (resp.$(-\epsilon, 0)$), a vertical ray $r'$ (resp, $r{''}$) is attached.  
Let $X$ be the space $X_0$ with these rays attached. The Morse boundary of $X$ is the discrete set of the vertical rays. 

The contracting property is easier to use than the Morse property in this case since $X$ is a $CAT(0)$ space. We can choose $\epsilon$ sufficiently small so that there are exactly three $2\epsilon$-contracting bi-infinite geodesics in $X$. They are $(r_0, r')$, $(r_0, r'')$ and $(r', r'')$. See Figure \ref{eg}. The constant $\epsilon$ is $\frac{1}{4}$.

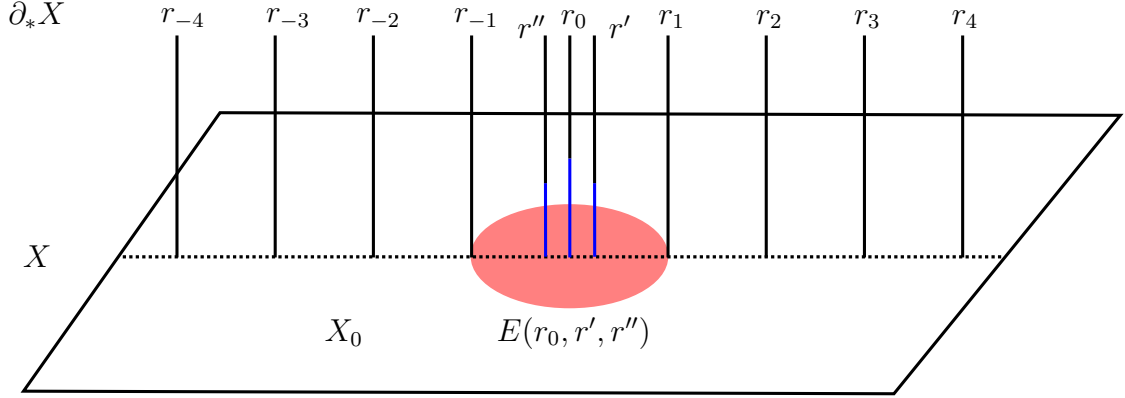
\begin{figure}[!ht]
\labellist

\pinlabel $\partial_*X$ at 10 105
\pinlabel $X$ at 10 42
\pinlabel $X_0$ at 90 22
\pinlabel $r_0$ at 150 104
\pinlabel $r_1$ at 175.5 104
\pinlabel $r_2$ at 201 104
\pinlabel $r_3$ at 226.5 104
\pinlabel $r_4$ at 252 104
\pinlabel $r_{-1}$ at 124.5 104
\pinlabel $r_{-2}$ at 99 104
\pinlabel $r_{-3}$ at 75.5 104
\pinlabel $r_{-4}$ at 48 104
\pinlabel $r'$ at 162 102
\pinlabel $r''$ at 139 102
\pinlabel $E(r_0,r',r'')$ at 150 22

\endlabellist
\includegraphics[width=6in]{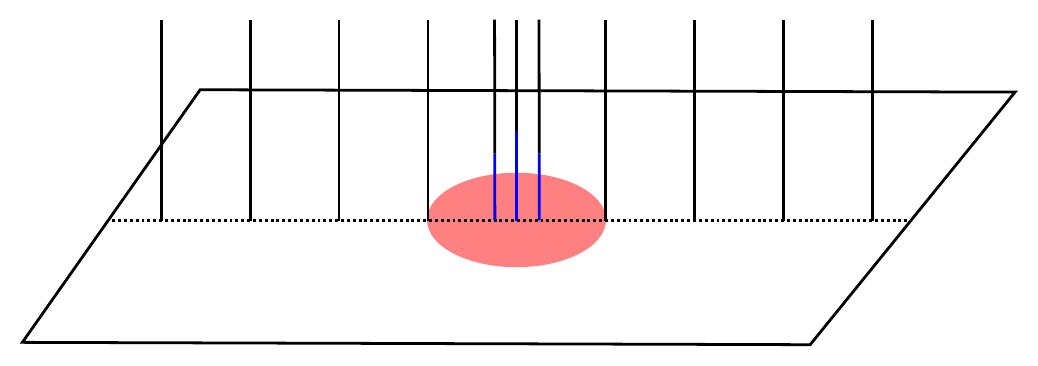}
\caption{$f$ is $\frac{1}{2}$-basetriangle stable but is not $2$-stable.} 
\label{eg}
\end{figure}

Construct a homeomorphism $f:\partial_*X\to \partial_*X$ as follows. $f(r_{2n})=r_{-2n}$ for any $n\in Z$ and $f$ fixes all other points in $\partial_*X$.
For any $n$, consider the geodesic $(r_{2n}, r_{2n+1})$ which is $1$-contracting, but the geodesic $(f(r_{2n}), f(r_{2n+1}))=(r_{-2n}, r_{2n+1})$ is $4n+1$-contracting. This implies $f$ is not a $2$-stable map. 

But we will see that $f$ is a $\frac{1}{2}$-basetriangle stable map. Note that the coarse center set $E(r_0, r', r'')$ is a closed unit ball centered at point $(0, 0)$ in $X$. More precisely, it is the union of three blue segments in the figure and the closed red unit ball in $X_0$ centered at $(0, 0)$. The map $f$ fixes the triangle $T(r_0, r', r'')$. For any nonnegative constant $C$ and $x_0, y_0\in E(r_0, r', r'')$, we have $f(\partial X_{x_0}^{(C)})\subset \partial X_{y_0}^{(C+1)}$. This means that $f$ is a $\frac{1}{2}$-basetriangle stable map. But for any constant $C\ge 1$, this map $f$ is not $C$-basetriangle stable.
\end{eg}

If we choose the points in the real line carefully, we have the following example. A basetriangle stable homeomorphism $f$ does not guarantee that $f$ is $2$-stable. 

\begin{eg}We can change the points in the above example. 
Let $X_0$ be a Euclidean plane $\bbR^2$. 
For any $m\in \bbZ_{\ge 0}$, at every point $(\frac{1}{2}m(m+1), 0)\in \bbZ^{2}\subset \bbR^{2}$, a vertical ray $r_m$ is attached. 
For any $m\in \bbZ_{\le 0}$, at points $(\frac{1}{2}m(1-m), 0)\in \bbZ^{2}\subset \bbR^{2}$, a vertical ray $r_m$ is attached. 
The space $X$ is $X_0$ together with these rays. As in the above example, the Morse boundary is the discrete set of the vertical rays.

Construct a homeomorphism $f:\partial_*X\to \partial_*X$ as follows. $f(r_{2n})=r_{-2n}$ for any $n\in Z$ and $f$ fixes all other points in $\partial_*X$. By a similar argument, the map $f$ is not a $2$-stable map. But it is basetriangle-stable. The reason is the following.

Given any positive constant $C$, there are only finitely many $C$-contracting bi-infinite geodesics in $X$. It follows that the union of the coarse center sets of all $(a, b, c)\in \partial_*X^{(3, C)}$ has bounded diameter. An analogous argument shows that the map $f$ is $C$-basetriangle stable. Thus, it is basetriangle stable and not $2$-stable.

\end{eg}

\subsection{Bihölder maps, quasisymmetries and strongly quasi-conformal maps}
Now we are ready to define bihölder maps, quasisymmetries and strongly quasi-conformal maps on the Morse boundaries. The motivations for these definitions come naturally from the fact that every quasi-isometry between two proper geodesic metric spaces indues a bihölder map, a quasisymmetry and a strongly quasi-conformal map on their Morse boundaries. This is shown at the end of this section.
\begin{defn}[Bihölder maps with respect to basepoints and metrics]\label{bihölder 1}
Let $X, Y$ be two Gromov hyperbolic spaces. Let $x_0\in X$, $y_0\in Y$ be  basepoints. Let $d_{x_0, \epsilon_X}$ and $d_{y_0, \epsilon_Y}$ be two metrics on the boundaries.
A map $f: (\partial_*X_{x_0}, d_{x_0, \epsilon_X})\to (\partial_*Y_{y_0}, d_{y_0, \epsilon_Y})$ is {\it bihölder} with respect to these metrics and basepoints, if there exist positive constants $C\ge 1, \alpha_1$, and $\alpha_2$ such that for all $x_1, x_2\in \partial_MX_{x_0}^{(N)}$, 
$$\frac{1}{C}d_{x_0, \epsilon_X}(x_1, x_2)^{\frac{1}{\alpha_1}}\le d_{y_0, \epsilon_{Y}}(f(x_1),f(x_2))^{\alpha_2}\le Cd_{x_0, \epsilon_{X}}(x_1, x_2)^{\alpha_1}.$$ 
We call $C, \alpha_1$ and $\alpha_2$ {\it bihölder constants} of $f$.
\end{defn}

\begin{rmk}
Bihölder maps can be considered between any two metric spaces. The constant $\alpha_2$ is usually taken to be $1$.
In our setting, the metric $d_{x_0, \epsilon}$ on the Gromov boundaries has a parameter $\epsilon$. 
From Remark \ref{B-eq}, for different $\epsilon$ and $\epsilon'$, the metrics $d_{x_0, \epsilon}$ and $d_{x_0, \epsilon'}$ are B-equivalent.
For this reason, we need a positive power $\alpha_2$ in the above definition.
\end{rmk}

Now let us give the definition of bihölder maps between Morse boundaries. We use notations as in the Definition \ref{tri sta}.

\begin{defn} [Bihölder maps between Morse boundaries]\label{bihölder 2}
Let $X$ and $Y$ be two proper geodesic metric spaces.  An $N_0$-basetriangle stable map $f: \partial_{*}X\to\partial_{*}Y$ is  {\it $(N_0, N)$-bihölder} if 
\begin{itemize}
\item $f_N: \partial_*{X^{(N)}}\hookrightarrow \partial_*{Y^{(N')}}$ is bihölder for all basetriangles $(a, b, c) \in \partial_{*}X^{(3, N_0)}$ and 
\item the bihölder constants depend only on $f$, $N$ and $N_0$.
\end{itemize}
We say that $f$ is {\it bihölder} if it is $(N_0, N)$-bihölder for every $N$ and every $N_0$.
\end{defn}

\begin{rmk}\label{bihölder}
In the case of hyperbolic spaces, a bihölder map $f$ between Gromov boundaries in the Definition  \ref{bihölder 1} is different from a bihölder map in the Definition \ref{bihölder 2}. The former one is defined in terms of the fixed basepoints and metrics. But the latter one is considering all basetriangles and the relevant constants does not depend on the choice of basetriangle.
\end{rmk}

Now let us give the definition of a quasisymmetric map. 
\begin{defn}
Let $(X_1, d_1), (X_2, d_2)$ be two metric spaces. A homeomorphism $f: X_1 \to X_2$ is said to be {\it quasisymmetric} if there exists an increasing homeomorphism $\psi_{f}: (0, \infty)\to (0, \infty)$ such that for any three distinct points $a, b, c\in X_1$ we have 
$$\frac{d_2(f(a), f(b))}{d_2(f(a), f(c))}\le \psi_{f}\left(\frac{d_1(a, b)}{d_1(a, c)}\right).$$
\end{defn}

Let $\alpha>0, \lambda \ge 1$.
The quasisymmetry $f$ is called a {\it power quasisymmetry} if the homeomorphism $\psi$ is given by the following

$$\psi(t)=\psi_{\alpha, \lambda}(t) = \left\{\begin{array}{l}\lambda t^{\frac{1}{\alpha}} \mbox{ if } 0<t<1\\ \lambda t^{\alpha} \mbox{ if } t\ge1\end{array}\right.$$

\begin{defn} [Quasisymmetries on Morse boundaries]
Let $X$ and $Y$ be proper geodesic metric spaces.  
An $N_0$-basetriangle stable homeomorphism $f: \partial_{*}X\to\partial_{*}Y$ is  {\it $(N_0, N)$-quasisymmetric } if 
\begin{itemize}
\item $f_N: \partial_*{X^{(N)}}\hookrightarrow \partial_*{Y^{(N')}}$ is quasisymmetric onto its image for all  basetriangles \\$(a, b, c) \in \partial_{*}X^{(3, N_0)}$ and 
\item the homeomorphism $\psi_{f_N}$ depend only on $f$, $N$ and $N_0$.
\end{itemize}
We say that $f$ is {\it quasisymmetric} if it is $(N_0, N)$-quasisymmetric for every $N$ and $N_0$. 
\end{defn}

Recall that in a metric space $(X, d)$, an $r$-{\it annulus} $A$ is defined by $$A=A(x_0, a, ar)=\{x\in X\ | \ a\le d(x, x_0)\le ar\},$$ where $r\ge 1$, $a> 0$ and $x_0$ is the center of the annulus. 
Sometimes we use $A(x_0, r)$ to emphasize the center and the ratio of radii of two concentric spheres.

P. Pansu gave the following notion of a quasi-conformal map which says that a quasi-conformal map does not distort annuli too much.

 \begin{defn}
 A map $f$ between two metric spaces $(X_1, d_1)$ and $(X_2, d_2)$ is said to be quasi-conformal if there exists a function $\phi_{f} : [ 1, \infty )\to [1, \infty )$ such that $f$ maps every $r$-annulus of $X_1$ into some $\phi_f(r)$-annulus of $X_2$.
 \end{defn}

In our setting, we will use a slightly different notion which we call {\it strongly quasi-conformal } maps, which take account of the center of the annulus.  

\begin{defn}
 A map $f$ between two metric spaces $(X_1, d_1)$ and $(X_2, d_2)$ is said to be {\it strongly quasi-conformal} if there exists a function $\phi_f : [ 1, \infty )\to [1, \infty )$ such that $f$ maps every $r$-annulus $A(x, r)$ of $X_1$ into some $\phi_f(r)$-annulus $A(f(x),\phi_f(r))$ of $X_2$.
 \end{defn}

\begin{defn} [Strongly quasi-conformal maps between Morse boundaries]
Let $X$ and $Y$ be two proper geodesic metric spaces.  
An $N_0$-basetriangle stable map $f: \partial_{*}X\to\partial_{*}Y$ is  {\it $(N_0, N)$-strongly quasi-conformal } if 
\begin{itemize}
\item $f_N: \partial_*{X^{(N)}}\to \partial_*{Y^{(N')}}$ is strongly quasi-conformal for all basetriangles \\$(a, b, c) \in \partial_{*}X^{(3, N_0)}$ and 
\item the function $\phi_{f_N}$ depends only on $f$, $N$ and $N_0$.
\end{itemize}
We say that $f$ is {\it strongly quasi-conformal} if it is $(N_0, N)$-strongly quasi-conformal for every $N$ and every $N_0$.
\end{defn}

 It is not hard to check the following.
Let $X, Y, Z$ be proper geodesic metric spaces. If $f: \partial_*X\to \partial_*Y$ and $g: \partial_*Y\to \partial_*Z$ are quasisymmetric (resp. bih\"older, strongly quasi-conformal), then $g\circ f: \partial_*X\to \partial_*Z$ is quasisymmetric (resp. bih\"older, strongly quasi-conformal).

It is helpful to know what happens to $N$-Morse Gromov products under quasi-isometries.
In the case of hyperbolic spaces, we have the following result which will be generalized to our case.
\begin{prop}
Given $K\ge 1, C, \delta\ge 0$, there is a constant $A=A(K, C, \delta)$ such that the following holds.

Let $X, Y$ be two geodesic $\delta$-hyperbolic metric spaces and let $h: X\to Y$ be a $(K, C)$-quasi-isometry. If $x_0, x_1, x_2\in X$, then 
$$K^{-1}(x_1\cdot x_2)_{x_0}-A\le (h(x_1) \cdot h(x_2))_{h(x_0)} \le K(x_1\cdot x_2)_{x_0}+A$$
\end{prop}

The next proposition is a special case of Proposition 3.18 in \cite{cordes2017stability}. 
\begin{prop}\label{QIGP}
Given a Morse gauge $N$ and constants $K\ge 1, C\ge 0$, there exist a constant $A=A(N, K, C)$ and a Morse gauge $N'$ such that the following holds. 

Let $X, Y$ be two proper geodesic metric spaces and let $h: X\to Y$ be a $(K, C)$-quasi-isometry, then $h$ induces an embedding $\partial_*h: \partial X_{x_0}^{(N)}\to \partial Y_{h(x_0)}^{(N')}$. For any $x_1, x_2 \in X_{x_0}^{(N)}$, we have
$$K^{-1}(x_1\cdot x_2)_{x_0}-A\le (h(x_1) \cdot h(x_2))_{h(x_0)} \le K(x_1\cdot x_2)_{x_0}+A.$$

In particular, there exists a constant $A'=A'(N, K, C)$ such that for any $x_1, x_2 \in \partial X_{x_0}^{(N)}$, we have 
$$K^{-1}(x_1\cdot_{N} x_2)_{x_0}-A'\le (\partial_*h(x_1) \cdot_{N'} \partial_*h(x_2))_{h(x_0)} \le K(x_1\cdot_{N} x_2)_{x_0}+A'.$$
\end{prop}

Note that the above proposition used the Gromov product with respect to the basepoints $x_0$ and $h(x_0)$. 
But we would like to use basetriangles in the Morse boundary instead of basepoints in the interior. 
The following basic proposition describes a connection between the basepoint and the coarse center of a basetriangle.
Fix a basetriangle $T$ on the Morse boundary $\partial_*X$, and choose a coarse center of $T$ as a basepoint in $X$. The next proposition tells us that,  under a quasi-isometry $h$, the image of this basepoint is not too far away from coarse centers of the image of this basetriangle $T$.
 
\begin{prop}\label{QI-bstri}
Let $N_0$ be a Morse gauge. For any constants $K\ge 1, C\ge 0$, there exists a constant $A=A(N_0, K, C)$ such that the following holds.

Let $X, Y$ be two proper geodesic metric spaces and let $h: X\to Y$ be a $(K, C)$-quasi isometry. Then for any $$(a, b, c)\in \partial_{*}X^{(3, N_0)}, x_0 \in E(a, b, c) \mbox{ and }
 y_0 \in E(\partial_*h(a), \partial_*h(b), \partial_*h(c)),$$ we have
$d_Y(h(x_0), y_0)\le A$.
\end{prop}
 
 \begin{proof}
 
By Proposition \ref{2=bstri}, the quasi isometry $h$ induces a basetriangle stable and $2$-stable homeomorphism $\partial_*h$ between their Morse boundaries. 
For any $(a, b, c)\in \partial_{*}X^{(3, N_0)}$, there exists $N_1$ such that $\partial_*h(\partial_*X^{(3, N_0)})\subset \partial_*Y^{(3, N_1)}$.
Since $x_0$ is a coarse center of $(a, b, c)$, then there exists a geodesic $\alpha$ joining $a$ and $b$, such that $d_X(x_0, \alpha)\le 1+\delta_{N_0}$. Since $h$ is a quasi-isometry, 
then $d_Y(h(x_0), h(\alpha))\le L_1$, where $L_1$ depends only on $K, C$ and $N_0$. Note that $h(\alpha)$ is a $(K, C)$-quasi geodesic between $\partial_*h(a)$ and $\partial_*h(b)$. 
Let $\alpha'$ be a geodesic between $\partial_*h(a)$ and $\partial_*h(b)$ in $Y$. By the Morse property of $\alpha'$, the Hausdorff distance between $h(\alpha)$ and $\alpha'$ is bounded by $L_2$, where $L_2$ depends only on $K, C$ and $N_1$. 
Thus $d_Y(h(x_0), \alpha')\le L_1+L_2$. It follows that $h(x_0)\in E_{L_1+L_2}(\partial_*h(a), \partial_*h(b), \partial_*h(c))$. 

Setting $L_0=\max\{L_1+L_2, 1+\delta_{N_1}\}$, we have $y_0, h(x_0)\in E_{L_0}(\partial_*h(a), \partial_*h(b), \partial_*h(c))$. From the Lemma \ref{coarse}, the set  $E_{L_0}(\partial_*h(a), \partial_*h(b), \partial_*h(c))$ has bounded diameter $A$, where $A$ depends only on $L_0$ and $N_1$. Thus $d_Y( h(x_0), y_0)\le A$ and $A=A(K, C, N_0)$.
 
 \end{proof}

We will now discuss the properties of the Gromov products with respect to coarse centers of basetriangles under quasi-isometries. 

\begin{prop}\label{key}
Let $N_0, N$ be Morse gauges. Let $K\ge 1, C\ge 0$. There exist a constant $A$ and a Morse gauge $N'$ depending only on  $N_0, N, K, C$ such that the following holds.

Let $X, Y$ be two proper geodesic metric spaces and let $h:X\to Y$ be a $(K, C)$-quasi-isometry. For any basetriangle and coarse centers 
$$(a, b, c)\in \partial_*X^{(3, N_0)}, x_0\in E(a, b, c), y_0\in E(\partial_*h(a), \partial_*h(b), \partial_*h(c)),$$
we have $$\partial_*h(\partial X_{x_0}^{(N)})\subset \partial Y_{y_0}^{(N')}.$$
And for any $x_1, x_2\in \partial X_{x_0}^{(N)}$, we have 
$$K^{-1}(x_1\cdot_{N} x_2)_{x_0}-A\le (\partial_*h(x_1) \cdot_{N'} \partial_*h(x_2))_{y_0} \le K(x_1\cdot_{N} x_2)_{x_0}+A.$$

\end{prop}

\begin{proof}
The first part comes from the fact that $h$ induces a basetriangle stable homeomorphism between their Morse boundaries. That is, for any $N$, there exists $N'=N'(K, C, N)$ such that $\partial_*h(\partial X_{x_0}^{(N)})\subset \partial Y_{y_0}^{N'}$. 
We can choose $N'$ sufficiently large, such that 
 $$\partial_*h(\partial X_{x_0}^{(N)})\subset \partial Y_{h(x_0)}^{(N')}.$$
This is possible since $d(y_0, h(x_0))\le A_0=A_0(K, C, N_0)$ by Proposition \ref{QI-bstri}. With Lemma \ref{dist-GP}, it is not hard to show that 
$$|(\partial_*h(x_1) \cdot_{N'} \partial_*h(x_2))_{y_0}-(\partial_*h(x_1) \cdot_{N'} \partial_*h(x_2))_{h(x_0)}|\le 2C_N'+A_0,$$
for all $x_1, x_2\in \partial X_{x_0}^{(N)}$, where $C_N'$ is a constant depending only on $N'$. 
The rest of the proof follows from Proposition \ref{QIGP}.

\end{proof}

 Now we are ready to show the following theorem.
 
 \begin{thm}
 Let $h: X\to Y$ be a $(K, C)$-quasi isometry between two proper geodesic metric spaces. Assume that $\partial_*X$ contains at least three points. Then the induced map $\partial_*h:\partial_*X\to \partial_*Y$ is a bihölder, quasisymmetric, strongly quasi-conformal homeomorphism. 
 \end{thm}

 \begin{proof}

The quasi isometry $h: X\to Y$ induces a homeomorphism $f:=\partial_*h$ between their Morse boundaries. 
Let $N_0$ and $N$ be Morse gauges. Then from Proposition \ref{key}, there exists $A$ and $N'$ such that the following holds. 

For any 
$$(a, b, c)\in \partial_*X^{(3, N_0)}, x_0\in E(a, b, c), y_0\in E(f(a), f(b), f(c)),$$
there exists a map $$f|_{\partial X_{x_0}^{(N)}}:(\partial X_{x_0}^{(N)}, d_{x_0, \epsilon_{N}})\to (\partial Y_{y_0}^{(N')}, d_{y_0,\epsilon_{N'}})$$ between two metric spaces.
And for any $x_1, x_2\in \partial X_{x_0}^{(N)}$, we have 

\begin{equation}\label{centerGP}
K^{-1}(x_1\cdot_{N} x_2)_{x_0}-A\le (f(x_1) \cdot_{N'} f(x_2))_{y_0} \le K(x_1\cdot_{N} x_2)_{x_0}+A.
\end{equation}

We will show that the map $f|_{\partial X_{x_0}^{(N)}}$ is bihölder, quasisymmetric and strongly quasi-conformal.
Equation \ref{key} implies that, 
$$e^{-K(x_1\cdot_{N} x_2)_{x_0}-A}\le e^{- (f(x_1) \cdot_{N'} f(x_2))_{y_0}}\le e^{-K^{-1}(x_1\cdot_{N} x_2)_{x_0}+A}  \iff$$
$$(e^{-\epsilon_{N}(x_1\cdot_{N} x_2)_{x_0}})^{K\frac{\epsilon_{N'}}{\epsilon_{N}}}\cdot e^{-A\epsilon_{N'}}\le e^{-\epsilon_{N'}(f(x_1) \cdot_{N'} f(x_2))_{y_0}}\le (e^{-\epsilon_{N}(x_1\cdot_{N} x_2)_{x_0}})^{K^{-1}\frac{\epsilon_{N'}}{\epsilon_{N}}}\cdot e^{A\epsilon_{N'}}.$$
From Theorem \ref{metric}, we have a constant $c=3-2\sqrt{2}$ such that
$$\displaystyle{ ce^{-\epsilon_N(x_1\cdot_N x_2)_{x_0}}\le d_{x_0, \epsilon_N}(x_1, x_2)\le e^{-\epsilon_N(x_1\cdot_N x_2)_{x_0}}},$$ $$\displaystyle{ ce^{-\epsilon_{N'}(f(x_1)\cdot_{N'} f(x_2))_{y_0}}\le d_{y_0, \epsilon_{N'}}(f(x_1), f(x_2))\le e^{-\epsilon_{N'}(f(x_1)\cdot_{N'} f(x_2))_{y_0}}}.$$
It is not hard to see that, there exist positive constants $C_0\ge 1, \alpha_1, \alpha_2$ such that the following holds.

$$\frac{1}{C_0}d_{x_0, \epsilon_N}(x_1, x_2)^{\frac{1}{\alpha_1}}\le d_{f(x_0), \epsilon_{N'}}(f(x_1), f(x_2))^{\alpha_2}\le C_0d_{x_0, \epsilon_N}(x_1, x_2)^{\alpha_1},$$
for all $x_1, x_2\in \partial_* X_{x_0}^{(N)}$. These constants $C_0, \alpha_1, \alpha_2$ depend only on $K, C, N, N_0$. 
So the map $f|_{\partial X_{x_0}^{(N)}}$ is bihölder.
This shows that $f:\partial_*X\to \partial_*Y$ is bihölder.

Next we show that $f|_{\partial X_{x_0}^{(N)}}$ is bihölder implies that $f|_{\partial X_{x_0}^{(N)}}$ is strongly quasi-conformal. 
By an easy argument, there exist a function $\phi :[1, \infty )\to  [1, \infty )$ and positive constants $\alpha=\alpha(K, C, N, N_0)>0$ and $\beta=\beta(K, C, N, N_0)\le 1$ such that the following holds.

For all $p_0\in \partial X_{x_0}^{(N)}$, $a> 0$ and $r\ge 1$, we have $$f|_{\partial X_{x_0}^{(N)}}(A(p_0, a, ra))\subset A(f(p_0), \beta a^{\alpha}, \phi(r)\beta a^{\alpha}),$$ where the $r$-annulus $A(p_0, a, ar)=\{p\in \partial X_{x_0}^{(N)} \ |\  a\le d_{x_0, \epsilon_{N}}(p, p_0)\le ar\}$ and the $f(r)$-annulus $A(f(p_0), \beta a^{\alpha}, \phi(r)\beta a^{\alpha})=\{q\in \partial X_{y_0}^{(N')} \ |\  \beta a^{\alpha}\le d_{y_0, \epsilon_{N'}}(q, f(p_0))\le \phi(r)\beta a^{\alpha}\}$. 
This shows that $f:\partial_*X\to \partial_*Y$ is strongly quasi-conformal.

Cordes and Hume \cite[Theorem 3.16]{cordes2017stability} show that, there is some $N_1=N_1(K, C, N)$ such that
$f:(\partial X_{x_0}^{(N)}, d_{x_0, \epsilon_N})\to (\partial Y_{h(x_0)}^{(N_1)}, d_{h(x_0), \epsilon_{N_1}})$ is a quasisymmetry onto its image.
We will change the basepoint $h(x_0)$ to $y_0$. By Proposition \ref{QI-bstri}, $d(h(x_0), y_0)\le A(K, C, N_0)$. By Lemma \ref{close M}, there exists $N'=N'(N_1, A)$ such that, changing the basepoint induces a natural map $i_{h(x_0), y_0}:\partial Y_{h(x_0)}^{(N_1)}\to \partial Y_{y_0}^{(N')}$. It is a quasi-symmetric map.
Thus $$f|_{\partial X_{x_0}^{(N)}}:(\partial X_{x_0}^{(N)}, d_{x_0, \epsilon_{N}})\to (\partial Y_{y_0}^{(N')}, d_{y_0,\epsilon_{N'}})$$ is a quasisymmetry. This shows that $f:\partial_*X\to \partial_*Y$ is bihölder.
\end{proof}

\begin{rmk}
Cordes and Hume actually show more:  the map $f: (\partial X_{x_0}^{(N)}, d_{x_0, \epsilon_{N}})\to (\partial Y_{y_0}^{(N')}, d_{y_0,\epsilon_{N'}})$ is a power quasi-symmetry onto its image and the power quasisymmetry constants depend only on $N, N', K, C$.
\end{rmk}

 \section{From Boundaries to the Interior}
 
 In this section we will prove the reverse implication in Theorem \ref{main} . That is, if $f:\partial_*X\to \partial_*Y$ is a homeomorphism such that $f$ and $f^{-1}$ are bih\"older or strongly quasi-conformal or quasi-symmetric, then $f$ is induced by a quasi-isometry between $X$ and $Y$. We assume that $\partial_*X$ contains at least three points and that  both $X$ and $Y$ are cocompact. By Proposition \ref{2=bstri}, we know that $f$ and $f^{-1}$ are $2$-stable maps in our setting.

 \subsection{From the Morse boundary to the interior}
Now given a homeomorphism $f$ between $\partial_*X$ and $\partial_*Y$, we would like to extend $f$ to a map $\Phi_{f}$ between the interiors. 

Let $\Theta_3(\partial_*X)$ be the set of distinct triples in $\partial_*X$. For any three distinct points $a, b, c\in\partial_*X$, we define a map:
 $$\pi_X:\Theta_3 (\partial_*X)\to X$$ 
 $$\pi_X((a, b, c))=p,$$ 
 where $p$ is a choice of coarse center in $E(a, b, c)$. We say $\pi_{X}((a, b, c))$ the {\it  projection} of $(a, b, c)$.

Now for any $N$, let us define a map: 
$$\pi_{X, N}:\partial_*X^{(3, N)}\to X$$
$$\pi_{X, N}((a, b, c))=\pi_{X}((a, b, c)),$$
where $(a, b, c)\in\partial_*X^{(3, N)}$. That is, $\pi_{X, N}=\pi_{X}|_{\partial_*X^{(3, N)}}$. Sometimes we want to emphasize the Morse gauge, we will use $\pi_{X, N}$ instead of $\pi_X$.

Fix a Morse gauge $N_0$ such that $\partial_*X^{(3, N_0)}\neq\emptyset$. Choose an ideal triangle $(a, b, c)\in \partial_*X^{(3, N_0)}$. Let $x_0=\pi_{X}((a, b, c))\in E(a, b, c)$. 
By hypothesis $X$ is cocompact. That is, there exists a group $G$ acting cocompactly by isometries on $X$. Isometries of $X$ preserve the Morse gauges of bi-infinite geodesics. 
Hence, for any $g\in G$, we have that $g(x_0)\in E(ga, gb, gc)$. By Lemma \ref{coarse}, the distance $$d_X(g(x_0), \pi_{X, N_0}(ga, gb, gc))\le L,$$ 
where $L$ depends on $N_0$.
Since $X$ is cocompact, there exists $R_0>0$ such that for any $x\in X$  and $g\in G$ such that $$d_X(g(x_0), x)\le R_0.$$ Let $R=R_0+L$, 
we  have that 
$$\pi_{X, N_0}^{-1}(B(x, R))\neq \emptyset$$ for every $x\in X$.
 
 We know that $f$ is $2$-stable, which implies $f(\partial_*X^{(3, N_0)})\subset\partial_*Y^{(3, N_1)}$ for some $N_1$. Let 
 $$\pi_Y: \Theta_3(\partial_*Y)\to Y\mbox{ and } \pi_{Y, N_1}:\partial_*Y^{(3, N_1)}\to Y $$ be the analogous map for $Y$, where $\pi_{Y, N_1}=\pi_Y|_{\partial_*Y^{(3, N_1)}}$. Now we define $$\Phi_f: X\to Y$$ by $\Phi_f(x)=y$, where $y$ is a point in the set 
 $$\pi_{Y, N_1}(f(\pi_{X, N_0}^{-1}(B(x, R))))\subset Y.$$
 This map $\Phi_{f}$ is called an {\it  extension} of $f$ from $\partial_*X$ to $X$. The definition of $\Phi_{f}$ depends on choices of $N_0, N_1$, $R$, $\pi_X$ and $\pi_Y$.  
 
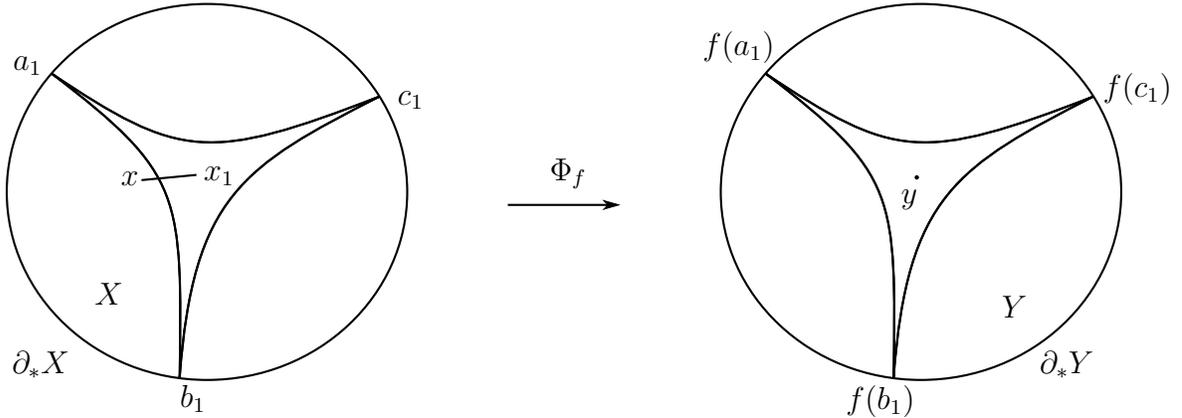
\begin{figure}[!ht]
\labellist

\pinlabel $\Phi_{f}$ at 220 85
\pinlabel $y$ at 350 75
\pinlabel $x$ at  53 83
\pinlabel $x_1$ at  87 84
\pinlabel $a_1$ at 14 126
\pinlabel $f(a_1)$ at 285 133
\pinlabel $b_1$ at 77 -1
\pinlabel $f(b_1)$ at 339 -2
\pinlabel $c_1$ at 160 113
\pinlabel $f(c_1)$ at 437 117
\pinlabel $X$ at 45 39
\pinlabel $Y$ at 390 34
\pinlabel $\partial_*X$ at 19 12
\pinlabel $\partial_*Y$ at 410 12

\endlabellist
\includegraphics[width=6in]{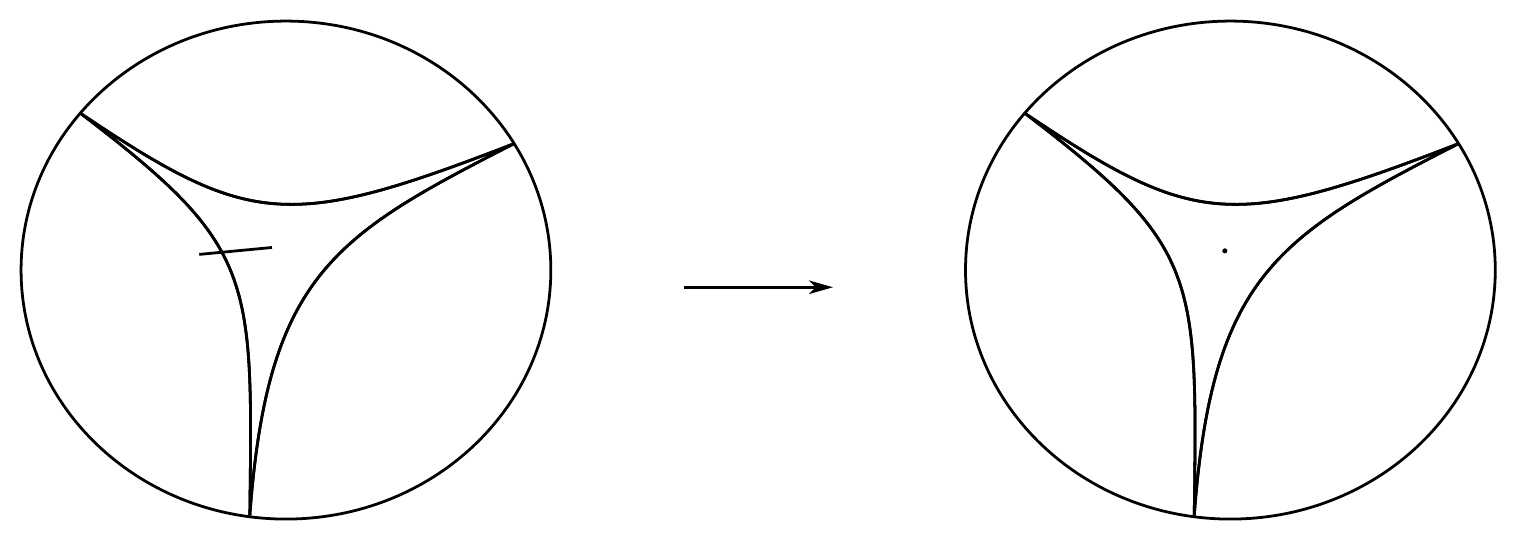}
\caption{$\Phi_{f}$ is an extension of $f$. $\Phi_{f}(x)=y$.} 
\label{extension}
\end{figure}

 From the above construction. It is easy to see that
 $$\Phi_f(x)=y$$ implies that there exists $(a_1, b_1, c_1)\in\partial_*X^{(3, N_0)}$ such that $$d_X(\pi_X((a_1, b_1, c_1)), x)\le R \mbox{ and } y=\pi_Y(f(a_1), f(b_1), f(c_1)).$$ See Figure \ref{extension}, where $x_1=\pi_{X}((a_1, b_1, c_1))$.
 
 We postpone the next proposition's proof till the end of this section. It is quite useful to control the function $\Phi_f$.

\begin{prop}\label{key prop}
Let $X, Y$ be two proper geodesic metric spaces. Suppose that $X$ is cocompact and $\partial_*X$ contains at least three points. Let $f:\partial_*X\to \partial_*Y$ be a homeomorphism which is quasisymmetric or bi-hölder or strongly quasi-conformal. Then for any Morse gauge $N$, there exists a function $\eta_N: (0, \infty)\to (0, \infty)$ such that for all $(a, b, c), (a', b', c')$ in $\partial_*X^{(3, N)}$, 
$$d_X(\pi_{X}(a, b, c), \pi_{X}(a', b', c'))\le \theta \Rightarrow 
d_Y(\pi_Y(f(a), f(b), f(c)), \pi_Y(f(a'), (b'), (c')) )\le \eta_N(\theta).$$

\end{prop}

Now suppose we have Proposition \ref{key prop}. 
\begin{prop}\label{control}
Fix $N_0, N_1$ and $R$ as above. Then for any $\upsilon\ge 0$, there exists $D$ such that for any $x_1, x_2\in X$, 
$$d_X(x_1, x_2)\le \upsilon \Rightarrow d_Y(\Phi_f(x_1), \Phi_f(x_2))\le D.$$
\end{prop}

\begin{proof}
For $i=1, 2$, choose $(a_i, b_i, c_i)\in \pi_{X, N_0}^{-1}(B(x_i, R))$ such that 
$$\pi_{Y}(f(a_i), f(b_i), f(c_i))=\Phi_f(x_i).$$
Since $d_X(x_1, x_2)\le \upsilon$ and $d_X(x_i, \pi_{X}(a_i, b_i, c_i))\le R$, we have 
 $$d_X(\pi_{X}(a_1, b_1, c_1), \pi_{X}(a_2, b_2, c_2))\le \upsilon+2R.$$
By Proposition \ref{key prop}, we choose $D=\eta_{N_0}(\upsilon+2R)$ and then 
$$d_Y(\pi_{Y}(f(a_1), f(b_1), f(c_1)), \pi_{Y}(f(a_2), (b_2), (c_2)))\le D.$$ 
This proves the proposition.

\end{proof}

\subsection{Main Theorem}

Now we are ready to prove our main theorem. It suffices to show the following theorem.

\begin{thm}\label{thm 2}
Let $X$ and $Y$ be proper, cocompact geodesic metric spaces and assume that $\partial_*X$ contains at least three points. Let $f: \partial_*X \to \partial_*Y$ be a homeomorphism. Suppose that $f$ and $f^{-1}$ are bihölder or quasi-symmetric or strongly quasi-conformal.
Then there exists a quasi-isometry $h: X \to Y$ with $\partial_*h=f$.
\end{thm}

\begin{proof}
We assume that, the Morse boundary of $X$ contains at least three points, similarly for $Y$. 
We can choose a Morse gauge $N_0$ so that 
$$\partial_*X^{(3, N_0)}\neq\emptyset\mbox{ and } \partial_*Y^{(3, N_0)}\neq\emptyset.$$
Note that by Proposition \ref{2=bstri}, homeomorphisms $f$ and $f^{-1}$ are $2$-stable.
There exists $N_1$ such that 
$$f(\partial_*X^{(2, N_0)})\subset \partial_*Y^{(2, N_1)} \mbox{ and } f^{-1}(\partial_*Y^{(2, N_0)})\subset \partial_*X^{(2, N_1)}.$$ 
By the construction above, there exists constant $R> 0$ so that 
$$\pi_{X, N_0}^{-1}(B(x, R))\neq\emptyset \mbox{ and } \pi_{Y, N_0}^{-1}(B(y, R))\neq\emptyset$$ for all $x\in X$ and $y\in Y$.

Recall that $\Phi_f$ is an extension of $f$ from $\partial_*X$ to $X$ and $\Phi_{f^{-1}}$ is an extension of $f^{-1}$ from $\partial_*Y$ to $Y$. 
Now let us prove that $\Phi_f$ is a quasi-isometry. 
It is enough to show the following.
\begin{itemize}
\item
There exist constants $A, B$ such that for any $x_1, x_2\in X$ and $y_1, y_2\in Y$, we have
$$d_Y(\Phi_{f}(x_1), \Phi_{f}(x_2))\le Ad_X(x_1,x_2)+B \mbox{ and } d_X(\Phi_{f^{-1}}(y_1), \Phi_{f^{-1}}(y_2))\le Ad_Y(y_1,y_2)+B.$$ 
\item
$\Phi_f$ and $\Phi_{f^{-1}}$ are quasi-inverses.
\end{itemize}

Let $x_1, x_2\in X$, choose integer $n$ such that $n\le d_X(x_1, x_2)< n+1$. Choose a sequence of $n+2$ points $x_1=p_0, p_1, ..., p_{n+1}=x_2$ on the geodesic $[x_1, x_2]$ such that 
$$d_X(p_i, p_{i+1})\le 1,$$ 
where  $0\le i\le n$. 
By Proposition \ref{control}, there exists $D=\eta_{N_0}(1+2R)$ such that 
for all $0\le i\le n$.
It follows that $$d_Y(\Phi_f(p_i), \Phi_f(p_{i+1}))\le D$$

$$d_Y(\Phi_f(x_1), \Phi_f(x_2))\le D(n+1)\le D(d_X(x_1, x_2)+1).$$
A similar argument for $\Phi_{f^{-1}}$ gives us that for all $y_1, y_2\in Y$, we have $$d_X(\Phi_{f^{-1}}(y_1), \Phi_{f^{-1}}(y_2))\le D(d_Y(y_1, y_2)+1).$$

Next we show that $\Phi_f$ and $\Phi_{f^{-1}}$ are quasi-inverses.
Let $y\in Y$. Set $ x=\Phi_{f^{-1}}(y) \mbox{ and } y_1=\Phi_f(x).$ 
Choose 
$(a_1, b_1, c_1)\in \pi^{-1}_{X, N_0}(B(x, R))\mbox { and } (f(a), f(b), f(c))\in \pi^{-1}_{Y, N_0}(B(y, R))$ such that 
$$\pi_{Y, N_1}(f(a_1), f(b_1), f(c_1))=y_1 \mbox{ and }\pi_{X, N_1}(a, b, c)=x.$$

Denote $y_0=\pi_Y(f(a), f(b), f(c))$ and $x_1=\pi_{X}(a_1, b_1, c_1)$.
See Figure \ref{4}.
Note that $$d_Y(y, y_0)\le R \mbox{ and } d_X(x, x_1)\le R.$$

\begin{figure}[!ht]
\labellist

\pinlabel $\Phi_{f}$ at 178 72
\pinlabel $\Phi_{f^{-1}}$ at 178 51
\pinlabel $y$ at 250 90
\pinlabel $x$ at  63 89
\pinlabel $y_0$ at 280 83
\pinlabel $x_1$ at 70 56
\pinlabel $a$ at 59 134
\pinlabel $a_1$ at 85 3
\pinlabel $f(a)$ at 266 136
\pinlabel $f(a_1)$ at 299 2
\pinlabel $b$ at 3 76
\pinlabel $b_1$ at 1.5 65
\pinlabel $f(b)$ at 208 76
\pinlabel $f(b_1)$ at 210 43
\pinlabel $c$ at 132 85
\pinlabel $c_1$ at 134 77
\pinlabel $f(c)$ at 350 83
\pinlabel $f(c_1)$ at 351 64
\pinlabel $y_1$ at 285 42
\pinlabel $X$ at 40 34
\pinlabel $Y$ at 310 34
\pinlabel $\partial_*X$ at 13 12
\pinlabel $\partial_*Y$ at 340 17

\endlabellist
\includegraphics[width=6in]{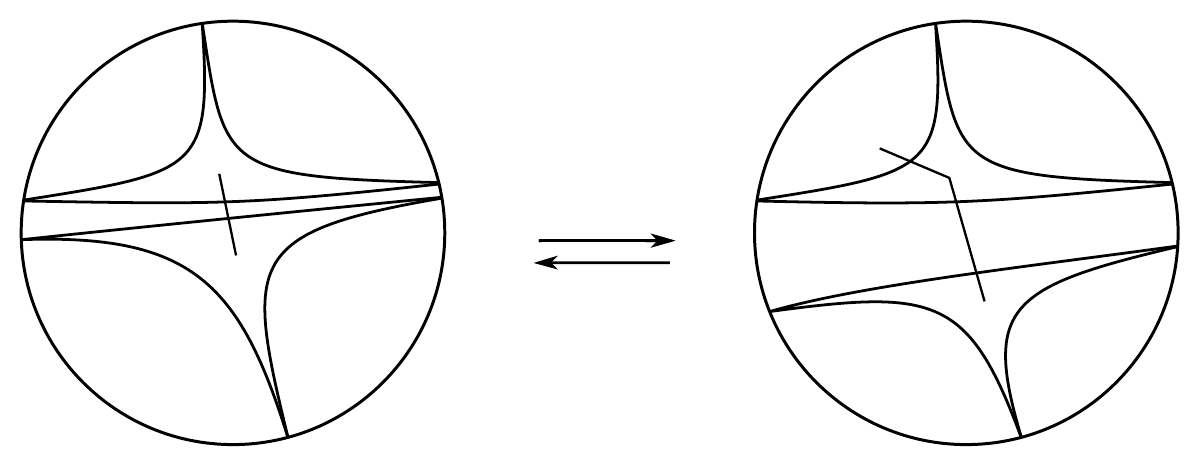}
\caption{ $d_Y(y, \Phi_{f}(\Phi_{f^{-1}}(y)))$ is bounded uniformally.} 
\label{4}
\end{figure}
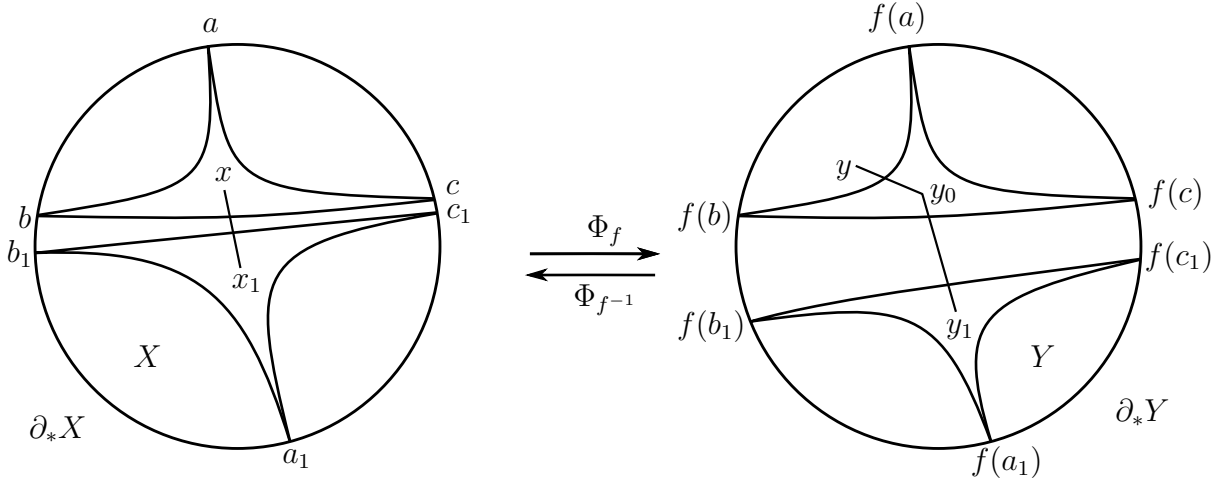

Let $N=\max \{N_0, N_1\}$.
By Proposition \ref{key prop}, the distance $$d_Y(y_0, y_1)\le\eta_{N}(R).$$

Then we have 
$$d_Y(y, y_1)\le d_Y(y, y_0)+ d_Y(y_0, y_1)\le R+\eta_{N}(R).$$
Hence, for any $y\in Y$, we have $$d_Y(y, \Phi_{f}(\Phi_{f^{-1}}(y)))\le R+\eta_{N}(R).$$
Using the same argument we have that $d_X(x, \Phi_{f^{-1}}(\Phi_{f}(x)))$ is bounded by $R+\eta_{N}(R)$ for all $x\in X$.

Thus $\Phi_{f}$ is a quasi-isometry. Next we will prove that the quasi-isometry $\Phi_{f}$ induces $f$ on the Morse boundary.
Choose a basepoint $x_0$ in $X$. Let $p$ be a point in $\partial_*X$. Let $\gamma$ be an $N$-Morse geodesic from $x_0$ to $p$ such that $\gamma(0)=x_0$. Let $x_n=\gamma(n)$. 
Setting $y_n=\Phi_{f}(x_n)$. Choose $y_0$ as a basepoint in $Y$.
By the construction of $\Phi_{f}$, for each $n$, there exists $(a_n, b_n, c_n)\in \pi^{-1}_{X, N_0}(B(x_n, R))$ such that $\pi_Y(f(a_n), f(b_n), f(c_n))=y_n.$
It is not hard to see that $x_n\in E_{\delta_{N_0}+1+R}(a_n, b_n, c_n)$. 

Using the slimness property of Morse triangles and Lemma \ref{close M},
we can find a Morse gauge $N'$ depending only on $R, N_0$ and $N$ such that all the points $x_n$, and all geodesics between $a_n, b_n$ and $c_n$ are subset of $X^{(N')}_{x_0}$, and the points $a_n, b_n, c_n\in\partial X_{x_0}^{(N')}\subset \partial_*^{N'}X_{x_0}$. 

From the compactness of $\partial_*^{N'}X_{x_0}$ and completeness of $\partial X_{x_0}^{(N')}$, passing to a subsequence, all these three sequences $\{a_n\}, \{b_n\}$ and $\{c_n\}$ converges to points $p_1, p_2$ and $p_3$ in $\partial X_{x_0}^{(N')}$, respectively. 

{\bf Claim :} Two of $p_1, p_2, p_3$ will be the point $p$.
\begin{proof}[Proof of Claim]
Firstly suppose that these three points are distinct, we have the coarse center set $E(p_1, p_2, p_3)$. 
When $n$ is sufficiently large, the sequences of points $\{a_n\}, \{b_n\}$ and $\{c_n\}$ are sufficiently close to points $p_1, p_2, p_3$ respectively in the visual metric $(\partial X_{x_0}^{(N')}, d_{x_0, \epsilon_{N'}})$. 
From Lemma \ref{close tri}, we can see that the point $x_n$ lies within a uniformly bounded distance from $E(p_1, p_2, p_3)$ when $n$ is sufficiently large. 
But the distance $d_X(x_n, x_0)$ goes to infinity. We get a contradiction. Without loss of generality, we may assume that $p_1=p_2$. 

Since $x_n\in E_{\delta_{N_0}+1+R}(a_n, b_n, c_n)$, there exists a geodesic $\alpha_n$ between $a_n$ and $b_n$ for every $n$ and a point $x_n'\in \alpha_n\subset X_{x_0}^{(N')}$ such that $$d_X(x_n, x_n')\le \delta_{N_0}+1+R.$$ 
Denote by $\alpha_n'$ be the segment of the geodesic $\alpha_n$ from $x_n'$ to $a_n$. 
We have $$d_X(x_0, \alpha_n)\le d_X(x_0, \alpha_n').$$
Note that $(a_n\cdot_{N'} b_n)_{x_0}\to \infty$ as $n\to \infty$. From Lemma \ref{dist-GP}, $d_X(x_0, \alpha_n)$ tends to infinity as $n\to \infty$. 
This implies that $$d_X(x_0, \alpha_n')\to \infty\mbox{ as } n\to \infty.$$
By Lemma \ref{dist-GP} again, we have $$(a_n\cdot_{N'} x_n')_{x_0}\to \infty\mbox{ as } n\to \infty.$$
Applying Lemma \ref{GP} twice, we get
$$(p_1\cdot_{N'} x_n)_{x_0}\ge\min \{(p_1\cdot_{N'} a_n)_{x_0}, (a_n\cdot_{N'} x_n')_{x_0}, (x_n'\cdot_{N'} x_n)_{x_0}\}-C_{N'},$$
where $C_{N'}$ is a constant depending only on $N'$.
It is easy to see that $(x_n'\cdot_{N'}x_n)_{x_0}$ and $(p_1\cdot_{N'}a_n)_{x_0}$ tend to infinity as $n\to \infty$. This implies that $p_1=\lim x_n$. Note that $p=\lim x_n$, so $p=p_1$. This proves the claim.
\end{proof}
Thus by passing to a subsequence, two sequences $\{a_n\}$ and $\{b_n\}$ converge to point the $p$ in the topology of $\partial X_{x_0}^{(N')}$. It follows they also converge to the point $p$ in the topology of $\partial_*X_{x_0}$.

Now consider two sequences $\{f(a_n)\}$ and $\{f(b_n)\}$, they converge to $f(p)$ in the topology of $\partial_*Y_{y_0}$ since $f$ is a homeomorphism between $\partial_*X_{x_0}$ and $\partial_*Y_{y_0}$. 
From Lemma \ref{conv}, we can find a Morse gauge $N_2$ such that all points $f(a_n), f(b_n), f(p)\in \partial Y_{y_0}^{(N_2)}$ for every $n$, and these two sequences $\{f(a_n)\}$ and $\{f(b_n)\}$ converge to $f(p)$ in the topology of $\partial Y_{y_0}^{(N_2)}$.  
Also for $N_2$ sufficiently large, an analogue of the argument used to prove the Claim shows that $f(p)=\lim y_n=\lim \Phi_f(x_n)$ in the topology of $\partial Y_{y_0}^{(N_2)}$. 
From Theorem \ref{s and *}, it is not hard to get that the Hausdorff distance between $\Phi_f(\gamma)$ and $\beta$ is finite, where $\beta$ is some geodesic from $y_0$ to $f(p)$.
Thus we conclude that $\partial_*\Phi_f(p)=f(p)$ for any $p\in \partial_*X$. This means that the quasi-isometry $\Phi_f: X\to Y$ induces the homeomorphism $f:\partial_*X\to \partial_*Y$.

\end{proof}

\subsection{Proof of Proposition \ref{key prop}}
\begin{proof}

Let $N$ be a Morse gauge and $(a_1, b_1, c_1), (a_2, b_2, c_2)\in \partial_*X^{(3, N)}$. 
As in Figure \ref{keyprop}, set 
$x_i=\pi_X((a_i, b_i, c_i))$ and $y_i=\pi_Y((f(a_i), f(b_i), f(c_i)))$, where $i=1, 2$.

If $d_X(x_1, x_2)\le \theta$, there exists $N_1=N_1(N, \theta)$
such that for any $i, j\in\{1,2\}$
$$a_i, b_i, c_i\in\partial X_{x_j}^{(N_1)}.$$
Since $f$ is basetriangle stable, there exists $N_2=N_2(N, f)$ so that for any $j=1, 2$, we have
$$f: \partial X_{x_{j}}^{(N_1)} \to \partial Y_{y_{j}}^{(N_2)}.$$
There are four metric spaces:
$$(\partial X_{x_{j}}^{(N_1)}, d_{x_{j}, \epsilon_{N_1}}),
(\partial Y_{y_{j}}^{(N_2)}, d_{y_{j}, \epsilon_{N_2}}), \mbox{ where } j=1, 2.$$ In the following argument, we will use $d_{x_j}$ and $d_{y_j}$ to simplify notations $d_{x_{j}, \epsilon_{N_1}}$ and $d_{y_{j}, \epsilon_{N_2}}$, where $j\in \{1, 2\}$.

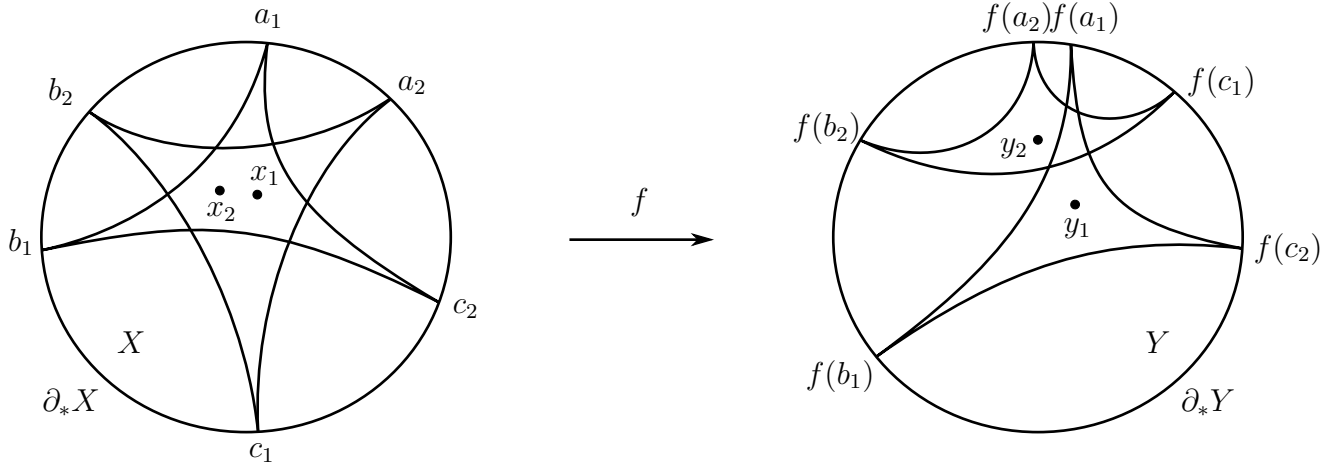
\begin{figure}[!ht]
\labellist

\pinlabel $a_1$ at 74 130
\pinlabel $a_2$ at 117 110
\pinlabel $b_1$ at 0 63
\pinlabel $b_2$ at 12 108
\pinlabel $c_1$ at 72 0
\pinlabel $c_2$ at 133 43
\pinlabel $f(a_1)$ at 318 130
\pinlabel $f(a_2)$ at 298 130
\pinlabel $f(b_1)$ at 245 23
\pinlabel $f(b_2)$ at 241 97
\pinlabel $f(c_1)$ at 359 111
\pinlabel $f(c_2)$ at 379 61
\pinlabel $X$ at 33 33
\pinlabel $Y$ at 340 33
\pinlabel $x_1$ at 73 83
\pinlabel $y_1$ at 316 67
\pinlabel $x_2$ at 60 72
\pinlabel $y_2$ at 297 91
\pinlabel $\partial_*X$ at 15 15
\pinlabel $\partial_*Y$ at 355 15
\pinlabel $f$ at 185 75

\endlabellist
\includegraphics[width=6.5in]{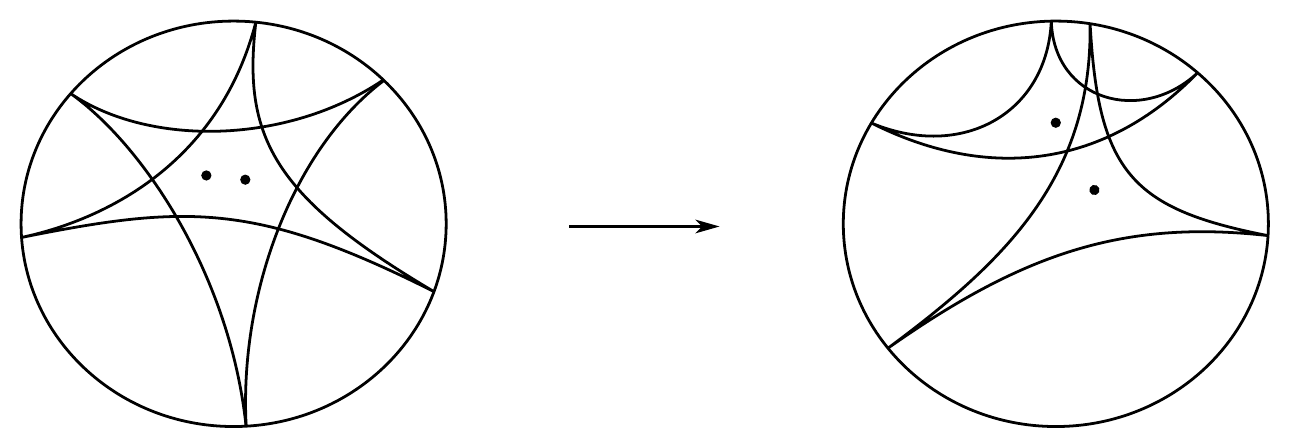}
\caption{The homeomorphism $f$ between the Morse boundaries $\partial_*X$ and $\partial_*Y$ induces a map $f: \partial X_{x_{j}}^{(N_1)} \to \partial Y_{y_{j}}^{(N_2)},$ where $j=1, 2$. } 
\label{keyprop}
\end{figure}

Since $x_{1}, x_{2}$ are coarse centers and distance between $x_1$ and $x_2$ is bounded by $\theta$, we have that both points $x_{1}$ and $x_{2}$ lie uniformly bounded distance $\theta+1+\delta_{N}$ from some geodesics $(a_i, b_i)$, $(a_i, c_i)$ and $(b_i, c_i)$ where $i= 1, 2$. 
Hence, by Lemma \ref{dist-GP}, there exists a positive constant $u=u(\epsilon_{N_1}, \delta_N, \theta)$ such that for any $i, j\in\{1, 2\}$,
\begin{equation}\label{lower bd1}
d_{x_j}(a_i, b_i), d_{x_j}(a_i, c_i),d_{x_j}(b_i, c_i)\ge u.
\end{equation}

For coarse center $y_{i}$, it lies uniformly bounded distance $1+\delta_{N}$ from some geodesics $(f(a_i), f(b_i))$, $(f(a_i), f(c_i))$ and $(f(b_i), f(c_i))$ where $i= 1, 2$. Again by Lemma \ref{dist-GP}, there exists positive constant $u'=u'(\epsilon_{N_2}, \delta_{N_2})$ such that for $j=1,2$, we have
\begin{equation}\label{lower bd2}
d_{y_j}(f(a_j), f(b_j)), d_{y_j}(f(a_j), f(c_j)),d_{y_j}(f(b_j), f(c_j))\ge u'.
\end{equation}
We would like to show a similar equation by switching the basepoints under certain conditions which is the following claim.

{\bf Claim:}
If $f$ is a bih\"older or quasisymmetric or strongly quasi-conformal homeomorphism, there exists a lower bound $u_0$ of $$d_{y_1}(f(a_2), f(b_2)), d_{y_1}(f(a_2), f(c_2)),d_{y_1}(f(b_2), f(c_2)),$$ where $u_0$ is a positive constant and depends only on $N, f$ and $\theta$.
\\

We will prove the claim later. Now with this claim, we are ready to find an upper bound for $d_Y(y_{1}, y_{2})$. 
Since $$d_{y_1}(f(a_2), f(b_2)), d_{y_1}(f(a_2), f(c_2)),d_{y_1}(f(b_2), f(c_2))\ge u_0>0,$$  
there exits a constant $\theta '=\theta '(u_0, N_2, \epsilon_{N_2})$ by Lemma \ref{dist-GP} so that
$$y_{1}\in E_{\theta '}((f(a_2), f(b_2), f(c_2))).$$
Let $\theta_1=\max\{\theta ', \delta_{N_2}\}$. We know that 
$$y_{1}, y_{2}\in E_{\theta_1}((f(a_2), f(b_2), f(c_2))). $$
Lemma \ref{coarse} says that the set $E_{\theta_1}((f(a_2), f(b_2), f(c_2)))$ has bounded diameter $L=L(\theta_1, N_2)$.
Finally,  setting $\eta_N(\theta)=L$, we get a constant, which depends only on $N, f$ and $\theta$, such that
$$d_Y(y_{1}, y_{2})\le \eta_N(\theta).$$
Now it remains to prove the claim.
\begin{proof}[Proof of Claim]
There are three cases.

{\bf Case one: When $f$ is a bih\"older homeomorphism.}

Suppose that $f$ is a bih\"older homeomorphism between $\partial_*X$ and $\partial_*Y$.
From the discussion above, we have a bih\"older map $$f: (\partial X_{x_{1}}^{(N_1)}, d_{x_1}) \to (\partial Y_{y_{1}}^{(N_2)}, d_{y_1})$$ and there exist positive constants $C\ge 1, \alpha_1$, and $\alpha_2$ such that 
$$\frac{1}{C}d_{x_1}(a, b)^{\frac{1}{\alpha_1}}\le d_{y_1}(f(a),f(b))^{\alpha_2},$$ for any $a, b\in \partial X_{x_{1}}^{(N_1)}$. Note that $C, \alpha_1$ and $\alpha_2$ depend only on $f, N_1$ and $N_2$.
From Equation \eqref{lower bd1}, we have 
$$d_{y_1}(f(a_2), f(b_2)), d_{y_1}(f(a_2), f(c_2)),d_{y_1}(f(b_2), f(c_2))\ge u_0,$$ where $u_0=(\frac{1}{C}u^{\frac{1}{\alpha_1}})^{\frac{1}{\alpha_2}}$ is a positive constant depending only on $f, N, \theta$.

{\bf Case two: When $f$ is a quasisymmetric homeomorphism.}

Now let $f$ be a quasisymmetric homeomorphism between $\partial_*X$ and $\partial_*Y$.

We have a quasisymmetric homeomorphism $$f: (\partial X_{x_{1}}^{(N_1)}, d_{x_1}) \to (\partial Y_{y_{1}}^{(N_2)}, d_{y_1})$$
and an increasing homeomorphism $\psi: (0, \infty)\to (0, \infty)$ such that for any three distinct points $a, b, c\in X_{x_{1}}^{(N_1)}$ we have 
$$\frac{d_{y_1}(f(a), f(b))}{d_{y_1}(f(a), f(c))}\le \psi\left(\frac{d_{x_1}(a, b)}{d_{x_1}(a, c)}\right).$$
The map $\psi$ depends only on $f$, $N_1$ and $N_2$.

There are three cases.
\begin{itemize}
\item If the set $\{a_1, b_1, c_1\}=\{a_2, b_2, c_2\}$, then the claim follows from Equation \ref{lower bd2}.

\item  If the set $\{a_1, b_1, c_1\}\cap\{a_2, b_2, c_2\}$ contains two elements. Without loss of generality, we may assume that $a_1=a_2, b_1=b_2, c_1\neq c_2$.

Note that $a_2, c_2, b_1$ are three distinct points. From Equation \eqref{lower bd1}, we have
$$\frac{d_{x_1}(c_2, b_1)}{d_{x_1}(c_2, a_2)}\le \frac{1}{u}.$$
Thus 
$$\frac{d_{y_1}(f(c_2), f(b_1))}{d_{y_1}(f(c_2), f(a_2))}\le \psi\left(\frac{d_{x_1}(c_2, b_1)}{d_{x_1}(c_2, a_2)}\right)\le\psi(\frac{1}{u}).$$

A similar argument for $a_2, c_2, c_1$, gives 
$$\frac{d_{y_1}(f(c_2), f(c_1))}{d_{y_1}(f(c_2), f(a_2))}\le \psi\left(\frac{d_{x_1}(c_2, c_1)}{d_{x_1}(c_2, a_2)}\right)\le\psi(\frac{1}{u}).$$

It follows that $$2d_{y_1}(f(c_2), f(a_2))\ge \frac{1}{\psi(\frac{1}{u})}(d_{y_1}(f(c_2), f(b_1))+d_{y_1}(f(c_2), f(c_1))).$$
With the triangle inequality and Equation \eqref{lower bd2}, we have $$d_{y_1}(f(c_2), f(b_1))+d_{y_1}(f(c_2), f(c_1))\ge d_{y_1}(f(b_1), f(c_1))\ge u'.$$ Finally we deduce that 
\begin{equation}\label{lower}
d_{y_1}(f(c_2), f(a_2))\ge \frac{u'}{2\psi(\frac{1}{u})}.
\end{equation}

Now consider $b_2, c_2, a_1$ and $b_2, c_2, c_1$, applying a similar proof, we get $$d_{y_1}(f(c_2), f(b_2))\ge \frac{u'}{2\psi(\frac{1}{u})}.$$
Since $a_1=a_2, b_1=b_2$ and from Equation \eqref{lower bd2} we get $$d_{y_1}(f(a_2), f(b_2))=d_{y_1}(f(a_1), f(b_1))\ge u'$$
Let $u_0=\min\{u', \frac{u'}{2\psi(\frac{1}{u})}\}$. We show that $$d_{y_1}(f(a_2), f(b_2)), d_{y_1}(f(a_2), f(c_2)),d_{y_1}(f(b_2), f(c_2))\ge u_0.$$

\item If the set $\{a_1, b_1, c_1\}\cap\{a_2, b_2, c_2\}$ contains at most one element.

By an analogous argument used to prove Equation \ref{lower}, we can show that 
$$d_{y_1}(f(a_2), f(b_2)), d_{y_1}(f(a_2), f(c_2)),d_{y_1}(f(b_2), f(c_2))\ge \frac{u'}{2\psi(\frac{1}{u})}.$$

\end{itemize}
Thus we conclude that $$d_{21}(f(a_2), f(b_2)), d_{21}(f(a_2), f(c_2)),d_{21}(f(b_2), f(c_2))\ge u_0,$$ where $u_0=\min\{u', \frac{u'}{2\psi(\frac{1}{u})}\}$ is a positive constant depending only on $N, f$ and $\theta$.

{\bf Case three: When $f$ is a strongly quasi-conformal  homeomorphism.}

Assume that $f$ is a strongly quasi-conformal  homeomorphism between $\partial_*X$ and $\partial_*Y$.

That is, there exists a strongly quasi-conformal map $$f: (\partial X_{x_{1}}^{(N_1)}, d_{x_1}) \to (\partial Y_{y_{1}}^{(N_2)}, d_{y_1})$$
and a function $\phi : [ 1, \infty )\to [1, \infty )$ such that $f$ maps every $r$-annulus $A(x, r)$ of $\partial X_{x_{1}}^{(N_1)}$ into some $\phi(r)$-annulus $A(f(x),\phi(r))$ of $\partial X_{y_{1}}^{(N_2)}$. The map $\phi$ depends only on $f$, $N_1$ and $N_2$. Recall that in a metric space $(X, d)$, an $r$-{\it annulus} $A$ is defined by $$A(x_0, r) =A(x_0, a, ar)=\{x\in X\ |\  a\le d(x, x_0)\le ar\}.$$ 
Now we will show that 
$$d_{y_1}(f(a_2), f(c_2)),d_{y_1}(f(b_2), f(c_2)), d_{y_1}(f(a_2), f(b_2))\ge u_0$$ for some positive constant $u_0$.

\begin{figure}[!ht]
\labellist

\pinlabel $f$ at 145 70
\pinlabel $c_2$ at 55 57
\pinlabel $a_2$ at 40 95
\pinlabel $b_2$ at 100 78
\pinlabel $a_1$ at 50 15
\pinlabel $b_1$ at 90 32
\pinlabel $A(c_2,\frac{u}{3},1)$ at 60 0

{\color {blue}
\pinlabel $1$ at 73 80
\pinlabel $\phi(\frac{3}{u})v$ at 247 85
}
{\color {red}
\pinlabel $\frac{u}{3}$ at 68 55
\pinlabel $v$ at 235 66
}

\pinlabel $f(c_2)$ at 222 54
\pinlabel $f(a_2)$ at 219 95
\pinlabel $f(b_2)$ at 261 68
\pinlabel $f(a_1)$ at 223 25
\pinlabel $f(b_1)$ at 248 37
\pinlabel $A(f(c_2),v,\phi(\frac{3}{u})v)$ at 224 0

\endlabellist
\includegraphics[width=6.5in]{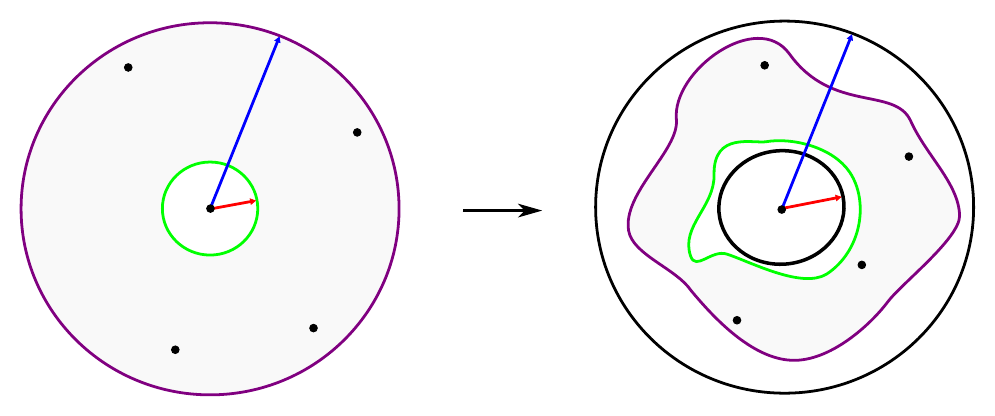}
\caption{Point $c_2$ is the center of a $\frac{3}{u}$-annulus and $f(A(c_2, \frac{u}{3}, 1))\subset A(f(c_2), v, \phi(\frac{3}{u})v)$.} 
\label{conformal}
\end{figure}
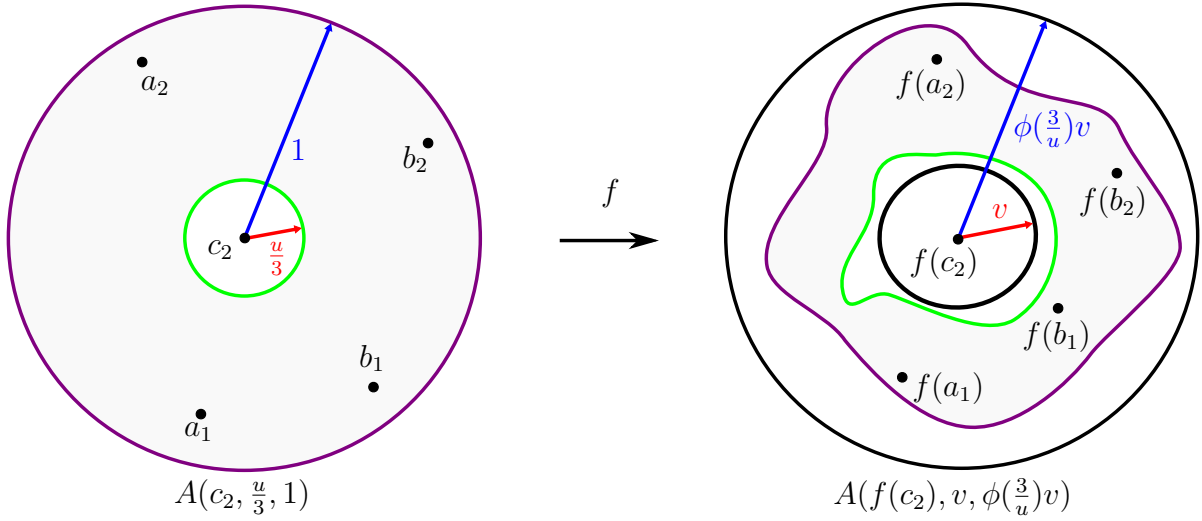

We regard the point $c_2$ as the center.
Let $u$ be the constant in Equation \eqref{lower bd1}.
Consider the $\frac{3}{u}$-annulus $$A(c_2, \frac{3}{u})=A(c_2, \frac{u}{3}, 1) \subset \partial X_{x_{1}}^{(N_1)}.$$
There exists some $\phi(\frac{3}{u})$-annulus $A(f(c_2), v, \phi(\frac{3}{u})v)\subset \partial X_{y_{1}}^{(N_2)}$ such that $$f(A(c_2, \frac{u}{3}, 1))\subset A(f(c_2), v, \phi(\frac{3}{u})v)$$ for some $v$.

Note that $a_2, b_2\in A(c_2, \frac{u}{3}, 1).$ It follows that 
$$f(a_2), f(b_2)\in A(f(c_2), v, \phi(\frac{3}{u})v).$$
Hence $$d_{y_1}(f(a_2), f(c_2)), d_{y_1}(f(b_2), f(c_2))\ge v.$$ It suffice to find a lower bound for $v$.

From Equation \eqref{lower bd1}, there are at most one point of $\{a_1, b_1, c_1\}$ in the ball $B(c_2, \frac{u}{3})$.
See Figure \ref{conformal}. Note that the diameter of $\partial_*X_{x_{1}}^{(N_1)}$ is bounded by $1$. Without loss of generality,  let us say $a_1, b_1\in A(c_2, \frac{u}{3}, 1)$.
Since $f(A(c_2, \frac{u}{3}, 1))\subset A(f(c_2), v, \phi(\frac{3}{u})v)$, we have $$d_{y_1}(f(a_1), f(c_2)), d_{y_1}(f(b_1), f(c_2))\le \phi(\frac{3}{u})v.$$

By the triangle inequality, we get 
$$d_{y_1}(f(a_1), f(b_1))\le d_{y_1}(f(a_1), f(c_2))+d_{y_1}(f(b_1), f(c_2))\le 2\phi(\frac{3}{u})r.$$
From Equation \eqref{lower bd2}, it is easy to see that
$$u'\le 2\phi(\frac{3}{u})v,$$ that is $$v\ge \frac{u'}{2\phi(\frac{3}{u})}.$$

Let $u_0=\frac{u'}{2\phi(\frac{3}{u})}$, we have shown that 
$$d_{y_1}(f(a_2), f(c_2)), d_{y_1}(f(b_2), f(c_2))\ge u_0.$$ 
If we regard the point $a_2$ as the center and use the similar argument as above, we will get that $$d_{y_1}(f(a_2), f(b_2))\ge u_0.$$

This proves the claim.
\end{proof}

\end{proof}

From the proofs in section 4, we also get the following theorem.  
\begin{thm}
Let $X$ and $Y$ be proper, cocompact geodesic metric spaces. 
Suppose that $\partial_*X$ contains at least three points.
Let $f: \partial_*X \to \partial_*Y$ be a homeomorphism. There exist Morse gauges $N_0$ and $N$ such that the following are equivalent.

\begin{enumerate}
\item $f$ is induced by a quasi-isometry $h: X \to Y$.
\item $f$ and $f^{-1}$ are $(N_0, N)$-bihölder.
\item $f$ and $f^{-1}$ are $(N_0, N)$-quasisymmetric.
\item $f$ and $f^{-1}$ are $(N_0, N)$-strongly quasi-conformal. 

\end{enumerate} 
\end{thm}


As noted in Remark \ref{bihölder}, our definition of bihölder is different from usual definition. Thus, it is interesting to have the following corollary in the case of hyperbolic spaces. One direction is not new for us. But the other direction seems to be new.
\begin{cor}
Let $X$ and $Y$ be proper, cocompact geodesic, hyperbolic spaces. 
Suppose that $\partial X$ contains at least three points.
Let $f: \partial X \to \partial Y$ be a homeomorphism. Then the following are equivalent.

\begin{enumerate}
\item $f$ is induced by a quasi-isometry $h: X \to Y$.
\item $f$ and $f^{-1}$ are bihölder.
\item $f$ and $f^{-1}$ are quasisymmetric.
\item $f$ and $f^{-1}$ are strongly quasi-conformal. 

\end{enumerate} 
\end{cor}


\begin{thebibliography}{TUK80b}

\bibitem[BH13]{bridson}
Martin~R Bridson and Andr{\'e} Haefliger.
\newblock {\em Metric spaces of non-positive curvature}, volume 319.
\newblock Springer Science \& Business Media, 2013.

\bibitem[BS00]{bonk2000embeddings}
M~Bonk and O~Schramm.
\newblock Embeddings of gromov hyperbolic spaces.
\newblock {\em Geometric And Functional Analysis}, 10(2):266--306, 2000.

\bibitem[CCM19]{charney2019quasi}
Ruth Charney, Matthew Cordes, and Devin Murray.
\newblock Quasi-mobius homeomorphisms of morse boundaries.
\newblock {\em Bulletin of the London Mathematical Society}, 51(3):501--515,
  2019.

\bibitem[CD16a]{C16}
Matthew Cordes and Matthew~Gentry Durham.
\newblock Boundary convex cocompactness and stability of subgroups of finitely
  generated groups.
\newblock {\em International Mathematics Research Notices}, 2016.

\bibitem[CD16b]{cordes2016boundary}
Matthew Cordes and Matthew~Gentry Durham.
\newblock Boundary convex cocompactness and stability of subgroups of finitely
  generated groups.
\newblock {\em International Mathematics Research Notices}, 2016.

\bibitem[CH17]{cordes2017stability}
Matthew Cordes and David Hume.
\newblock Stability and the morse boundary.
\newblock {\em Journal of the London Mathematical Society}, 95(3):963--988,
  2017.

\bibitem[CK00]{croke2000spaces}
Christopher~B Croke and Bruce Kleiner.
\newblock Spaces with nonpositive curvature and their ideal boundaries.
\newblock {\em Topology}, 39(3):549--556, 2000.

\bibitem[CM19]{cashen2019metrizable}
Christopher Cashen and John Mackay.
\newblock A metrizable topology on the contracting boundary of a group.
\newblock {\em Transactions of the American Mathematical Society},
  372(3):1555--1600, 2019.

\bibitem[Cor17]{Cordes}
Matthew Cordes.
\newblock Morse boundaries of proper geodesic metric spaces.
\newblock {\em Groups, Geometry, and Dynamics}, 11(4):1281--1306, 2017.

\bibitem[CS14]{CS2014}
Ruth Charney and Harold Sultan.
\newblock Contracting boundaries of {CAT} (0) spaces.
\newblock {\em Journal of Topology}, 8(1):93--117, 2014.

\bibitem[Ghy90]{ghys1990groupes}
{\'E}tienne Ghys.
\newblock Sur les groupes hyperboliques d'apres mikhael gromov.
\newblock {\em Progr. Math.}, 83, 1990.

\bibitem[Gro87]{gromov1987hyperbolic}
Mikhael Gromov.
\newblock Hyperbolic groups.
\newblock In {\em Essays in group theory}, pages 75--263. Springer, 1987.

\bibitem[Led94]{ledrappier1994structure}
Fran{\c{c}}ois Ledrappier.
\newblock Structure au bord des vari{\'e}t{\'e}s {\`a} courbure n{\'e}gative.
\newblock {\em S{\'e}minaire de th{\'e}orie spectrale et g{\'e}om{\'e}trie},
  13:97--122, 1994.

\bibitem[Liu19]{liu2019dynamics}
Qing Liu.
\newblock Dynamics on the morse boundary.
\newblock {\em arXiv preprint arXiv:1905.01404}, 2019.

\bibitem[MR19]{mousley2019hierarchically}
Sarah~C Mousley and Jacob Russell.
\newblock Hierarchically hyperbolic groups are determined by their morse
  boundaries.
\newblock {\em Geometriae Dedicata}, 202(1):45--67, 2019.

\bibitem[Mur19]{Murray}
Devin Murray.
\newblock Topology and dynamics of the contracting boundary of cocompact {CAT}
  (0) spaces.
\newblock {\em Pacific Journal of Mathematics}, 299(1):89--116, 2019.

\bibitem[Ota92]{otal1992geometrie}
Jean-Pierre Otal.
\newblock Sur la g{\'e}ometrie symplectique de l'espace des g{\'e}od{\'e}siques
  d'une vari{\'e}t{\'e} {\`a} courbure n{\'e}gative.
\newblock {\em Revista matem{\'a}tica iberoamericana}, 8(3):441--456, 1992.

\bibitem[Pan89]{pansu1989dimension}
Pierre Pansu.
\newblock Dimension conforme et spherea l?infini des vari{\'e}t{\'e}sa courbure
  n{\'e}gative.
\newblock {\em Ann. Acad. Sci. Fenn. Ser. AI Math}, 14(2):177--212, 1989.

\bibitem[Pau96]{paulin1996groupe}
Fr{\'e}d{\'e}ric Paulin.
\newblock Un groupe hyperbolique est d{\'e}termin{\'e} par son bord.
\newblock {\em Journal of the London Mathematical Society}, 54(1):50--74, 1996.

\bibitem[QRT19]{qing2019sublinearly}
Yulan Qing, Kasra Rafi, and Giulio Tiozzo.
\newblock Sublinearly morse boundary i: Cat (0) spaces.
\newblock {\em arXiv preprint arXiv:1909.02096}, 2019.

\bibitem[QZ19]{qing2019rank}
Yulan Qing and Abdul Zalloum.
\newblock Rank one isometries in sublinearly morse boundaries of cat (0)
  groups.
\newblock {\em arXiv preprint arXiv:1911.03296}, 2019.

\bibitem[TUK80a]{tukia1980two}
P~TUKIA.
\newblock On two dimensional quasiconformal groups.
\newblock {\em Ann. Acad. Sci. Fenn. Ser. AI Math.}, 5:73--78, 1980.

\bibitem[TUK80b]{tukia1980quasisymmetric}
P~TUKIA.
\newblock Quasisymmetric embeddings of metric spaces.
\newblock {\em Ann. Acad. Sci. Fenn. Ser. AI Math}, 5:97--114, 1980.

\bibitem[TV82]{tukia1982quasiconformal}
Pekka Tukia and J~Vaisala.
\newblock Quasiconformal extension from dimension n to n+ 1.
\newblock {\em Annals of Mathematics}, 115(2):331--348, 1982.

\bibitem[TV84]{tukia1984bilipschitz}
Pekka Tukia and Jussi V{\"a}is{\"a}l{\"a}.
\newblock Bilipschitz extensions of maps having quasiconformal extensions.
\newblock {\em Mathematische Annalen}, 269(4):561--572, 1984.

\bibitem[V{\"a}i81]{vaisala1981quasisymmetric}
Jussi V{\"a}is{\"a}l{\"a}.
\newblock Quasisymmetric embeddings in euclidean spaces.
\newblock {\em Transactions of the American Mathematical Society},
  264(1):191--204, 1981.

\bibitem[V{\"a}i06]{vaisala2006lectures}
Jussi V{\"a}is{\"a}l{\"a}.
\newblock {\em Lectures on n-dimensional quasiconformal mappings}, volume 229.
\newblock Springer, 2006.

\end{thebibliography}

\end{document}